\documentclass[10pt]{article}
\usepackage{fullpage}
\usepackage{amsmath}
\usepackage{amsthm}
\usepackage{amssymb}
\usepackage{wasysym}
\usepackage{pifont}
\usepackage{graphicx}
\usepackage{psfrag}
\usepackage{datetime}
\usepackage{bbm}
\usepackage{tikz}
\usepackage{tikz-cd}
\usepackage{titling}
\usepackage{subfigure}
\usepackage[utf8]{inputenc} 
\usetikzlibrary{positioning,arrows}
\tikzset{
  state/.style={circle,draw,minimum size=6ex},
  arrow/.style={-latex, shorten >=1ex, shorten <=1ex}}

\usepackage{geometry}
\usepackage[all]{xy}
\usepackage{graphicx}
\usepackage{titling}
\usepackage{fancyhdr}
\usepackage{array}
\usepackage{capt-of}
\usepackage[mathlines]{lineno}

\DeclareMathOperator{\Conf}{Conf}
\DeclareMathOperator{\Auto}{Aut}

\usepackage{hyperref}

\begin{document}
	\title{The boundary algebra of a GL$_m$-dimer}
	\author{
		Lukas Andritsch\thanks{Mathematics and Scientific Computing, University of Graz, Graz, Austria, \newline  {\tt lukas.andritsch@uni-graz.at} }}

	\newtheorem{lm}{Lemma}[section]
	\newtheorem{prop}[lm]{Proposition}
	\newtheorem*{claim}{Claim}
	\newtheorem{notation}[lm]{Notation}
	\newtheorem*{corollary}{Corollary}
	\newtheorem{cor}[lm]{Corollary}
	\newtheorem{theorem}[lm]{Theorem}
	\newtheorem*{thm}{Theorem}
	
	\theoremstyle{definition}
	\newtheorem{defn}[lm]{Definition}
	\newtheorem{def_con}[lm]{Definition and construction}
	\newtheorem*{defini}{Definition}
	\newtheorem*{definitionen}{Definitionen}
	
	\newtheorem*{rem}{Remark}
	\newtheorem{remark}[lm]{Remark}
	
	\maketitle
	
	
	\begin{abstract}
		We consider GL$_m$-dimers of triangulations of regular convex $n$-gons, which give rise to a dimer model with boundary $Q$ and a dimer algebra $\Lambda_Q$. Let $e_b$ be the sum of the idempotents of all the boundary vertices, and $\mathcal{B}_Q:= e_b \Lambda_Q e_b$ the associated boundary algebra. In this article we show that given two different triangulations $T_1$ and $T_2$ of the $n$-gon, the boundary algebras are isomorphic, i.e. $e_b \Lambda_{Q_{T_1}} e_b \cong e_b \Lambda_{Q_{T_2}} e_b$.
	\end{abstract}
	
	\textit{Keywords:} boundary algebra, dimer model, quiver, triangulation 
	
	\textit{2010 MSC:} 16G20, 82B20, 57Q15
	
	\tableofcontents


	\section{Introduction} 
	Dimer models with boundary were introduced by Baur, King and Marsh in \cite{bmt}. In case without boundary the definition is similar to dimer models defined by Bocklandt \cite{bl}. Dimer models with boundary are quivers with faces satisfying certain axioms. To any dimer model $Q$ one can associate its dimer algebra $A_Q$ as the path algebra of $Q$ modulo the relations arising from an associated potential.\\
	A source for dimer models are Postnikov diagrams of type $(k,n)$ in the disk, introduced by Postnikov in \cite{p} and used in \cite{bmt} as a combinatorial approach to Grassmannian cluster categories. In general, dimer algebras arising from different $(k,n)$ diagrams are not isomorphic. However, if you consider their boundary algebra, which is the idempotent subalgebra $B_Q:= e A_Q e$, where $e:= e_1 + \ldots + e_t$ is the sum of all idempotents corresponding to the boundary vertices, then one of the main results of \cite{bmt} is that for any two $(k,n)$-diagrams the associated boundary algebras are isomorphic.
	\ \\
	
	In this article, we study another source for dimer models, the so-called GL$_m$-dimers. They arise from Goncharov's $A_{m-1}^{*}$-webs on disks defined in \cite{g}. 
	In the case when $S$ is a disk with $n$ special points on the boundary, $A_{m}^{*}$-webs describe a cluster coordinate systems on the moduli space $\Conf_n(\mathcal{A}_{SL_{m+1}}^{*})$ as shown in \cite{g}. This moduli space is defined as $n$-tuples of the moduli space $GL_{m+1}/U$ of all decorated flags in an $m+1$-dimensional vector space $V_{m+1}$, where $U$ is the upper triangular unipotent subgroup in $GL_{m+1}$, modulo the diagonal action of the group $SL_m=\Auto(V_m,\Omega_m)$, where $\Omega_m$ is a volume form in $V_m$.\newline 
	\ \\
	The purpose of this article is to show that on the disk the boundary algebra of any two different GL$_m$-dimers are all isomorphic. \\
	The paper is structured as follows. After giving the necessary background in Section 2, we will first prove this result for the boundary algebra of a GL$_2$-dimer in Section $3$ and address the general case in Section $4$. \\
	For GL$_2$-dimers, the main result is the following: the quiver of the boundary algebra of any GL$_2$-dimer on an $n$-gon is given by $\Gamma(n)$, the following quiver shown in Figure \ref{fig:general_boundary_algebra}.
	\vspace{-2 cm}		\vspace{-0.3 cm}
	\begin{center}\includegraphics[width=220pt]{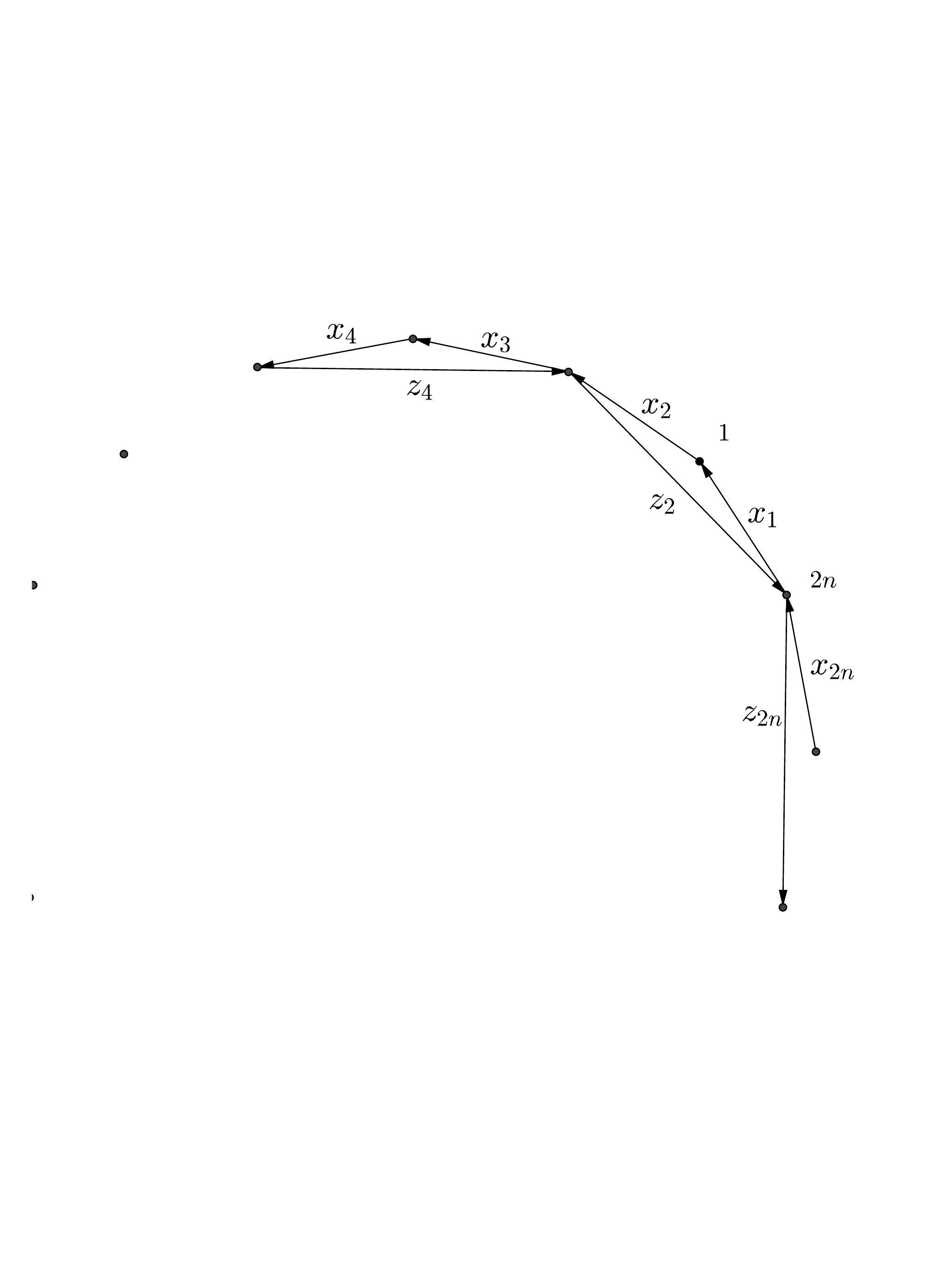}
		\vspace{-3.1 cm}
		\captionof{figure}{Part of the quiver $\Gamma(n)$. \label{fig:general_boundary_algebra}} 
	\end{center}
	
	Any two $3$-cycles incident with a common vertex are equivalent. Furthermore, any composition $z_{2k} z_{2k-2}$ is equivalent to the corresponding composition of $2n-4$ arrows $x_{2k+1} \cdots x_{2k-2}$, reducing modulo $2n$ and considering the composition of paths from left to right.\\
	The strategy to prove this result is to first prove it for boundary algebras arising from fan triangulations and then to show flip invariance.\\
	Throughout this paper, when we consider indices modulo $k$, we always assume them to be between $1$ and $k$. In particular, $0$ is never used as an index.


	\section{Settings and the GL$_m$-dimer}\label{sec:gen_set}
	
	
	\subsection{Background}
	\begin{defn}[quiver with faces]
		A quiver with faces is a quiver $Q=(Q_0,Q_1)$ together with a set $Q_2$ of faces and a map 
		\begin{align*}
		\partial: Q_2 \rightarrow Q_{cyc},
		\end{align*} 
		which assigns to each $F \in Q_{2}$ its boundary $\partial F \in Q_{cyc}$, where $Q_{cyc}$ is the set of oriented cycles in $Q$ (up to cyclic equivalence).
	\end{defn}
	We will always denote a quiver with faces by $Q$, regarded now as the tuple $(Q_0,Q_1,Q_2,s,t)$. A quiver with faces is called  finite if $Q_0$,$Q_1$ and $Q_2$ are finite sets. The (unoriented) \emph{incidence graph} of $Q$, at a vertex $i \in Q_0$, has vertices given by the arrows incident with $i$. The edges between two arrows $\alpha$,$\beta$ correspond to the paths of the form 
	\begin{figure}[!htb]
		\begin{center}
			\begin{tikzpicture}[scale=1]
			\node(a) at (0,-1){};
			\node(i) at (1.5,-1){};
			\node(b) at (3,-1){};
			
			\node(i1)[label=above: $i$] at (1.5,-1.125){};
			\node(a1)[label=above: $\alpha$] at (0.75,-1.125){};
			\node(b1)[label=above: $\beta$] at (2.25,-1.2){};	
			\path[line width=0.25mm,->] (a) edge (i);
			\path[line width=0.25mm,->] (i) edge (b);
			
			\draw[fill] (i) circle (1pt);
			\end{tikzpicture}
			\vspace{-0.6 cm}	
		\end{center}
		
	\end{figure}\\
	occurring in a cycle bounding a face.  
	\begin{defn}[dimer model with boundary \cite{bmt}]\label{def:dimer_model}
		A (finite, oriented) dimer model with boundary is given by a finite quiver with faces $Q=(Q_0,Q_1,Q_2)$ where $Q_2$ is written as disjoint union $Q_2=Q_2^{+} \cup Q_{2}^{-}$, satisfying the following properties:
		\begin{itemize}
			\item[(a)] the quiver $Q$ has no loops, i.e. no $1$-cycles, but $2$-cycles are allowed,
			\item[(b)] all arrows in $Q_1$ have face multiplicity $1$ (boundary arrows) or $2$ (internal arrows),
			\item[(c)] each internal arrow lies in a cycle bounding a face in $Q_2^{+}$ and in a cycle bounding a face in $Q_{2}^{-}$,
			\item[(d)] the incidence graph of $Q$ at each vertex is connected.
		\end{itemize}
	\end{defn}
	Note that, by (b), each incidence graph in (d) must be either a line (at a boundary vertex)
	or an unoriented cycle (at an internal vertex).
	
	
	\subsection{The dimer algebra and the boundary algebra}\label{subsec:dimer_algebra_and_boundary_algebra} 
	
	\begin{defn}[natural potential $W$]
		Let $Q=(Q_0,Q_1,Q_2)$ be a dimer model with boundary. Then the following formula
		\begin{align*}
		W:=W_Q:= \sum_{\gamma \in Q_{2}^{+}}{\partial \gamma}-\sum_{\gamma  \in Q_{2}^{-}}{\partial \gamma}
		\end{align*} 
		defines the natural potential associated to $Q$.
	\end{defn}
	\begin{rem}[differentiation of $W$] \
		\begin{figure}[!htb]
			\begin{center}
				\includegraphics[scale=0.65]{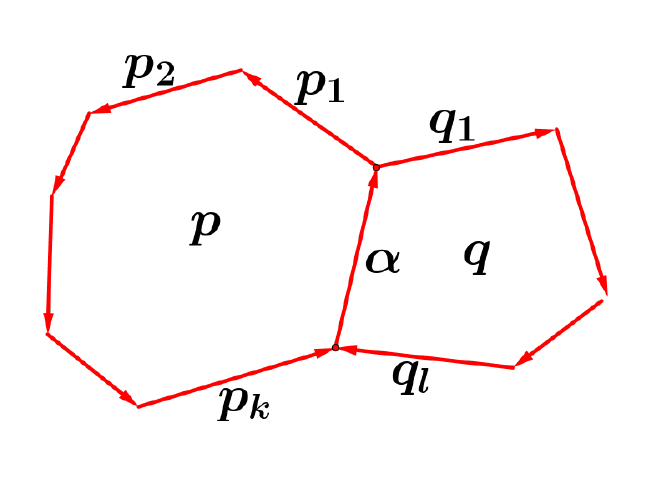}
				\vspace{-0.5 cm}
				\caption{$\alpha$ is part of a positive cycle $p$ and a negative cycle $q$.\label{fig:natural_potential}}
			\end{center}
		\end{figure} 
		Let $\partial W$ be the set of all cyclic derivatives with respect to all internal arrows $\alpha$ in $Q$. That means, if $\alpha$ is both part of the negative (clockwise) cycle $q=\alpha q_1\ldots q_l$ and the positive (counterclockwise) cycle $p=\alpha p_1 \ldots p_k$ with $k,l \geq 1$ as shown in Figure \ref{fig:natural_potential}, then the equation
		\begin{align*}
		\frac{\partial W}{\partial \alpha}: p_1 p_2 \ldots p_k=q_1 q_2 \ldots q_l 
		\end{align*}
		holds. In this article we use the notation $p_1 p_2 \ldots p_k \overset{\alpha}{\cong} q_1 q_2 \ldots q_l$ for relations obtained by the natural potential $W$.
	\end{rem}
	\begin{defn}[dimer algebra]
		Let $Q=(Q_0,Q_1,Q_2)$ be a dimer model with boundary and let $W$ and $\partial W$ be defined as above. Then the dimer algebra $\Lambda_Q$ is defined as
		\begin{align*}
		\Lambda_Q :=\mathbb{C}Q/_{\langle \partial W \rangle} .
		\end{align*}
	\end{defn}
	
	As usual, we write $e$ to denote an idempotent of an algebra and in the path algebra $\mathbb{C}Q$, let $e_i$ be the trivial path of length zero at vertex $i$. It is an idempotent of $\mathbb{C}Q$.
	Define
	$$e_b := e_1 + \ldots + e_t$$
	where $1,\ldots,t$ are the boundary vertices of the quiver; i.e. the vertices that are incident with boundary arrows. Furthermore, we call the remaining vertices of the quiver \textit{internal} (or \textit{inner}) \textit{vertices} and all arrows, that are incident with at least one internal vertex are called \textit{internal} (or \textit{inner}) \textit{arrows}.
	
	\begin{defn}[boundary algebra]
		The boundary algebra of a dimer model $Q$ with boundary is the spherical subalgebra consisting of linear combinations of paths which have starting and terminating points on the boundary of the quiver (i.e. one of the idempotent elements $e_1,\ldots e_t$):
		\begin{align*}
		\mathcal{B}_Q:= e_b \Lambda_Q e_b.
		\end{align*}
	\end{defn}
	
	
	\subsection{The GL$_m$-dimer} \
	
	\begin{defn}[triangulation]
		A triangulation of a regular convex polygon is a subdivision of the $n$-gon by diagonals into triangles, where each pair of diagonals  intersects at most in one of the vertices of the polygon.
	\end{defn} 
	\begin{rem}
		Every triangulation of an $n$-gon uses $n-3$ diagonals.
	\end{rem}
	\begin{rem}
		A special case of triangulation is the so called \textit{fan triangulation}, where each diagonal of the triangulation contains a given fixed vertex of the polygon.
	\end{rem}
	
	We recall Goncharov's definition of bipartite graphs $\Gamma_{A_{m-1}^{*}}(T)$ of an $m$-triangulation of a decorated surface $S$ as in \cite{g}. We use these graphs\footnote{up to a changed condition at the boundary} to define a family of dimer models with boundary.
	\begin{def_con}[GL$_m$-dimer]
		
		Take an arbitrary triangulation of the polygon. Every triangle is subdivided with $(m-1)$-lines in equidistance parallel to each of its sides, as in Figure \ref{fig:GL_m-dimer} (left) for $m=4$.
		Each triangle of the triangulation is now subdivided into small triangles of two kinds, namely \textit{upwards} and \textit{downwards} triangles w.r.t. an arbitrary edge. The subdivided triangle in the case of $m=4$ consists of $10$ upwards and $6$ downwards triangles for example.
		
		\begin{figure}[!htb]
			\begin{center}
				\begin{tikzpicture}[scale=1]
				\node(a) at (-1,-1){};
				\node(aab) at (-0.25,-1){};
				\node(aab) at (1.25,-1){};
				\node(b) at (2,-1){};
				\node(c) at (5,-1){};
				\node(d) at (0.5,1.12){};
				\node(e) at (2,3.24){};
				\node(f) at (3.5,1.12){};
				\node(ab) at (0.5, -1){};
				\node(bc) at (3.5, -1){};
				\node(ad) at (-0.25,0.06){};
				\node(de) at (1.25,2.18){};
				\node(ef) at (2.75,2.18){};
				\node(cf) at (4.25,0.06){};
				\draw[dotted] (a) to (c);
				\draw[dotted] (ab) to (ad);
				\draw[dotted] (bc) to (de);
				\draw[dotted] (ad) to (cf);
				\draw[dotted] (de) to (ef);
				\draw[dotted] (ab) to (ef);
				\draw[dotted] (bc) to (cf);
				\draw[dotted] (a) to (e);
				\draw[dotted] (e) to (c);
				\draw[dotted] (d) to (b);
				\draw[dotted] (d) to (f);
				\draw[dotted] (f) to (b);
				\node(k) at (7,-1){};
				\node(kkl) at (7.75,-1){};
				\node(kl) at (8.5, -1){};
				\node(kll) at (9.25,-1){};
				\node(l) at (10,-1){};
				\node(llm) at (10.75,-1){};
				\node(lm) at (11.5, -1){};
				\node(lmm) at (12.25,-1){};
				\node(m) at (13,-1){};
				\node(n) at (8.5,1.12){};
				\node(nno) at (8.875,1.65){};
				\node(no) at (9.25,2.18){};
				\node(noo) at (9.625,2.71){};
				\node(o) at (10,3.24){};
				\node(oop) at (10.375,2.71){};
				\node(op) at (10.75,2.18){};
				\node(opp) at (11.125,1.65){};
				\node(p) at (11.5,1.12){};
				\node(kkn) at (7.375,-0.47){};
				\node(kn) at (7.75,0.06){};
				\node(knn) at (8.125,0.59){};
				\node(mmp) at (12.625,-0.47){};
				\node(mp) at (12.25,0.06){};
				\node(mpp) at (11.875,0.59){};
				\node(np) at (10,1.12){};
				\node(h1) at (9.25,0.06){};
				\node(h2) at (10.75,0.06){};
				\node(u1) at (10,2.53){};
				\node(u2) at (9.25,1.473){};
				\node(u3) at (10.75,1.473){};
				\node(u4) at (8.5,0.413){};
				\node(u5) at (10,0.413){};
				\node(u6) at (11.5,0.413){};
				\node(u7) at (7.75,-0.65){};
				\node(u8) at (9.25,-0.65){};
				\node(u9) at (10.75,-0.65){};
				\node(u0) at (12.25,-0.65){};
				\node(d1) at (10,1.827){};
				\node(d2) at (9.25,0.77){};
				\node(d3) at (10.75,0.77){};
				\node(d4) at (8.5,-0.29){};
				\node(d5) at (10,-0.29){};
				\node(d6) at (11.5,-0.29){};
				\draw[fill=white] (u1) circle (3pt);
				\draw[fill=white] (u2) circle (3pt);
				\draw[fill=white] (u3) circle (3pt);
				\draw[fill=white] (u4) circle (3pt);
				\draw[fill=white] (u5) circle (3pt);
				\draw[fill=white] (u6) circle (3pt);
				\draw[fill=white] (u7) circle (3pt);
				\draw[fill=white] (u8) circle (3pt);
				\draw[fill=white] (u9) circle (3pt);
				\draw[fill=white] (u0) circle (3pt);
				\draw[fill] (kkl) circle (2pt);
				\draw[fill] (kll) circle (2pt);
				\draw[fill] (llm) circle (2pt);
				\draw[fill] (lmm) circle (2pt);
				\draw[fill] (kkn) circle (2pt);
				\draw[fill] (knn) circle (2pt);
				\draw[fill] (nno) circle (2pt);
				\draw[fill] (noo) circle (2pt);
				\draw[fill] (oop) circle (2pt);
				\draw[fill] (opp) circle (2pt);
				\draw[fill] (mmp) circle (2pt);
				\draw[fill] (mpp) circle (2pt);
				\draw[fill] (d1) circle (2pt);
				\draw[fill] (d2) circle (2pt);
				\draw[fill] (d3) circle (2pt);
				\draw[fill] (d4) circle (2pt);
				\draw[fill] (d5) circle (2pt);
				\draw[fill] (d6) circle (2pt);
				\draw (u1) -- (noo);
				\draw (u1) -- (oop);
				\draw (u1) -- (d1);
				\draw (u2) -- (nno);
				\draw (u2) -- (d1);
				\draw (u2) -- (d2);
				\draw (u3) -- (opp);
				\draw (u3) -- (d1);
				\draw (u3) -- (d3);
				\draw (u4) -- (knn);
				\draw (u4) -- (d2);
				\draw (u4) -- (d4);
				\draw (u5) -- (d2);
				\draw (u5) -- (d3);
				\draw (u5) -- (d5);
				\draw (u6) -- (d3);
				\draw (u6) -- (mpp);
				\draw (u6) -- (d6);
				\draw (u7) -- (kkn);
				\draw (u7) -- (kkl);
				\draw (u7) -- (d4);
				\draw (u8) -- (d4);
				\draw (u8) -- (kll);
				\draw (u8) -- (d5);
				\draw (u9) -- (d5);
				\draw (u9) -- (llm);
				\draw (u9) -- (d6);
				\draw (u0) -- (d6);
				\draw (u0) -- (lmm);
				\draw (u0) -- (mmp);
				\draw[dotted] (k) to (m);
				\draw[dotted] (kl) to (kn);
				\draw[dotted] (lm) to (no);
				\draw[dotted] (kn) to (mp);
				\draw[dotted] (no) to (op);
				\draw[dotted] (kl) to (op);
				\draw[dotted] (lm) to (mp);
				\draw[dotted] (k) to (o);
				\draw[dotted] (o) to (m);
				\draw[dotted] (n) to (l);
				\draw[dotted] (n) to (p);
				\draw[dotted] (p) to (l);
				\end{tikzpicture}
				\vspace{-0.3 cm}
				\caption{Constructing the GL$_m$-dimer on a triangle. Here $m=4$. \label{fig:GL_m-dimer}}
			\end{center} 
		\end{figure}
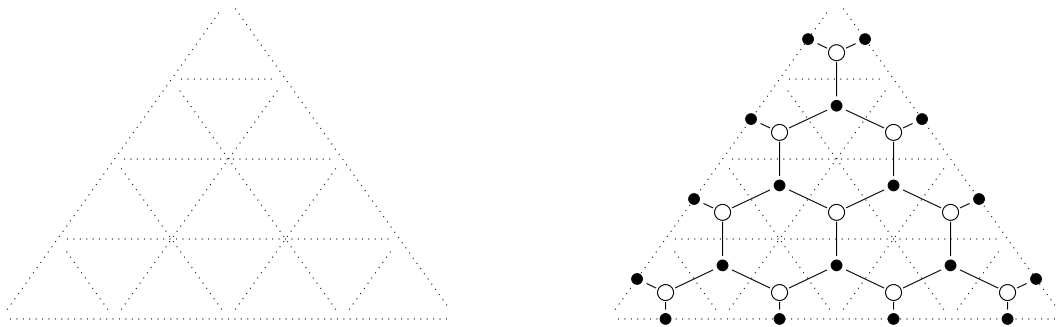\ \\
		%
		%
		%
		From such a subdivision we create a bipartite graph: We apply the following procedure to each triangle of the triangulation (see Figure \ref{fig:GL_m-dimer} (right)).
		\begin{itemize}
			\item Put black points on the midpoints of the short segments of the sides of the original triangle (e.g. either diagonals of the triangulation or edges of the polygon) and put black points into every downwards triangle.
			\item  Put a white point inside every upwards triangle.
			
		\end{itemize}
		Finally two points are connected if they differ in color and the points belong to the same small triangle or their small triangles have a side in common. 
		
		We call the resulting graph a \textit{GL$_m$-dimer}.
	\end{def_con}
	According to the first point of this list, there are exactly $m$  black points on each of the diagonals of the triangulation and on the edges of the initial polygon. Figure \ref{fig:dimer_pentagon} shows  a GL$_2$-dimer of a triangulated pentagon. 
	
	\begin{figure}[!htb]
		\begin{center}
			\begin{tikzpicture}[scale=1]
			\node(a)[label=above:$1$] at (5,5){};
			\node(b)[label=left:$2$] at (2.78,3.36){};
			\node(c)[label=left:$3$] at (3.65371, 0.74187){};
			\node(d)[label=right:$4$] at (6.4137, 0.76377){};
			\node(e)[label=right:$5$] at (7.24575, 3.39544){};
			\node(ab) at (3.89, 4.18){};
			\node(bc) at (3.21686, 2.05093){};
			\node(cd) at (5.03371, 0.75282){};
			\node(de) at (6.82973, 2.07961){};
			\node(ae) at (6.12288, 4.19772){};
			\node(aab) at (4.445, 4.59){};
			\node(abb) at (3.335, 3.77){};
			\node(bbc) at (2.99843, 2.70547){};
			\node(bcc) at (3.43529, 1.3964){};
			\node(ccd) at (4.34371, 0.74734){};
			\node(cdd) at (5.7237, 0.7583){};
			\node(dde) at (6.62171, 1.42169){};
			\node(dee) at (7.03774, 2.73752){};
			\node(aee) at (6.68431, 3.79658){};
			\node(aae) at (5.56144, 4.59886){};
			\node(u11) at (3.24567, 3.24522){};
			\node(u12) at (3.6063, 1.90788){};
			\node(u13) at (4.29751, 3.9214){};
			\node(d1) at (3.86175, 3.01982){};
			\node(u21) at (4.89856, 3.53072){};
			\node(u22) at (4.31253, 1.45709){};
			\node(u23) at (5.63485, 1.47212){};
			\node(d2) at (4.92861, 2.28354){};
			\node(u31) at (6.68669, 3.36543){};
			\node(u32) at (5.6649, 3.93643){};
			\node(u33) at (6.43124, 1.9229){};
			\node(d3) at (6.16077, 3.10998){};
			\node(m1) at (4.32686, 2.87093){};
			\node(m2) at (5.70685, 2.88189){};
			\node(aac) at (4.66343, 3.93547){};
			\node(acc) at (3.99029, 1.8064){};
			\node(aad) at (5.35343, 3.94094){};
			\node(add) at (6.06028, 1.82283){};
			\draw[fill=white] (u11) circle (3pt);
			\draw[fill=white] (u12) circle (3pt);
			\draw[fill=white] (u13) circle (3pt);
			\draw[fill=white] (u21) circle (3pt);
			\draw[fill=white] (u22) circle (3pt);
			\draw[fill=white] (u23) circle (3pt);
			\draw[fill=white] (u31) circle (3pt);
			\draw[fill=white] (u32) circle (3pt);
			\draw[fill=white] (u33) circle (3pt);
			\draw[fill] (d1) circle (2pt);
			\draw[fill] (d2) circle (2pt);
			\draw[fill] (d3) circle (2pt);
			\draw[fill] (aab) circle (2pt);
			\draw[fill] (abb) circle (2pt);
			\draw[fill] (bbc) circle (2pt);
			\draw[fill] (bcc) circle (2pt);
			\draw[fill] (ccd) circle (2pt);
			\draw[fill] (cdd) circle (2pt);
			\draw[fill] (dde) circle (2pt);
			\draw[fill] (dee) circle (2pt);
			\draw[fill] (aae) circle (2pt);
			\draw[fill] (aee) circle (2pt);
			\draw[fill] (aac) circle (2pt);
			\draw[fill] (acc) circle (2pt);
			\draw[fill] (aad) circle (2pt);
			\draw[fill] (add) circle (2pt);
			\draw[fill=gray] (a) circle (2pt);
			\draw[fill=gray] (b) circle (2pt);
			\draw[fill=gray] (c) circle (2pt);
			\draw[fill=gray] (d) circle (2pt);
			\draw[fill=gray] (e) circle (2pt);
			\draw (u11) -- (abb);
			\draw (u11) -- (bbc);
			\draw (u11) -- (d1);
			\draw (u12) -- (d1);
			\draw (u12) -- (bcc);
			\draw (u12) -- (acc);
			\draw (u13) -- (aab);
			\draw (u13) -- (d1);
			\draw (u13) -- (aac);
			\draw (u21) -- (aac);
			\draw (u21) -- (aad);
			\draw (u21) -- (d2);
			\draw (u22) -- (d2);
			\draw (u22) -- (ccd);
			\draw (u22) -- (acc);
			\draw (u23) -- (cdd);
			\draw (u23) -- (d2);
			\draw (u23) -- (add);
			\draw (u31) -- (aee);
			\draw (u31) -- (dee);
			\draw (u31) -- (d3);
			\draw (u32) -- (d3);
			\draw (u32) -- (aae);
			\draw (u32) -- (aad);
			\draw (u33) -- (dde);
			\draw (u33) -- (d3);
			\draw (u33) -- (add);
			\draw  (a) -- (b);
			\draw[dotted] (ab) to (bc);
			\draw[dotted] (ae) to (de);
			\draw[dotted] (ab) to (cd);
			\draw[dotted] (bc) to (m1);
			\draw[dotted] (m1) to (m2);
			\draw[dotted] (ae) to (cd);
			\draw[dotted] (m2) to (de);
			\draw[dotted] (m2) to (ae);
			\draw  (a) -- (e);
			\draw (e) -- (d);
			\draw (d) -- (c);
			\draw (c) -- (b);
			\draw[dotted] (a) to (c);
			\draw[dotted] (a) to (d);
			\end{tikzpicture}
			\vspace{-0.3 cm}
			\caption{A GL$_2$-dimer of a pentagon.}\label{fig:dimer_pentagon}
		\end{center}
	\end{figure}
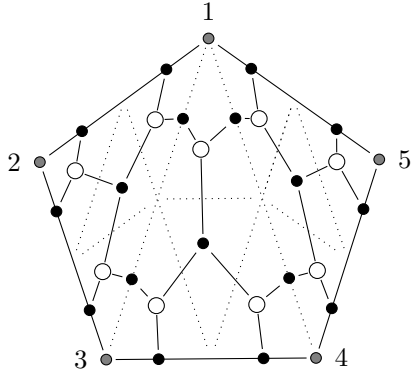 
	
	The resulting graph is bipartite and its complement splits the original surface into several connected components.
	
	
	\subsection{The boundary algebra of a GL$_m$-dimer} 
	
	We can associate a dimer model with boundary to a GL$_m$-dimer: 
	
	Put a vertex in each connected component of the complement of the GL$_m$-dimer.
	
	Then  connect adjacent components by arrows such that the white point of the dimer is on the left hand side of the arrow, shown in Figure \ref{fig:arrow_quiver}. Since the GL$_m$-dimer is bipartite, the quiver arising from it is a dimer model with boundary; each white vertex sits in a counterclockwise face and each black vertex in a clockwise face.
	Note that $Q_2^{+}$ and $Q_{2}^{-}$ are the set of all faces whose boundaries are oriented counterclockwise and clockwise respectively.
	
	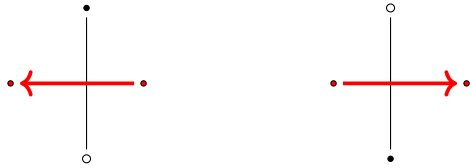
\begin{figure}[!htb]
		\begin{center}
			\begin{tikzpicture}[scale=0.5]
			\node(a) at (2,0){};
			\node(b) at (2,4){};
			\node(c) at (10,0){};
			\node(d) at (10,4){};
			\node(e) at (0,2){};
			\node(f) at (3.5,2){};
			\node(g) at (8.5,2){};
			\node(h) at (12,2){};
			\draw[fill=white] (a) circle (3pt);
			\draw[fill=white] (d) circle (3pt);
			\draw[fill=red] (e) circle (2pt);
			\draw[fill=red] (f) circle (2pt);
			\draw[fill=red] (g) circle (2pt);
			\draw[fill=red] (h) circle (2pt);
			\draw[fill] (b) circle (2pt);
			\draw[fill] (c) circle (2pt);
			\draw (a) -- (b);
			\draw (c) -- (d);
			\path[red,line width=0.5mm,->] (f) edge (e);
			\path[red,line width=0.5mm,->] (g) edge (h);
			\end{tikzpicture}
			\vspace{-0.3 cm}
			\caption{The white point of the dimer is on the left hand side of the arrow.\label{fig:arrow_quiver}}
		\end{center}
	\end{figure} 
	The quiver $Q$ of the GL$_m$-dimer of a triangle fulfills all aspects of Definition \ref{def:dimer_model} and hence it is a dimer model with boundary\footnote{We will often use the notation quiver instead of dimer model with boundary for readability of the article.}. An example of the quiver of the GL$_2$-dimer of a triangle is shown in Figure \ref{fig:quiver_triangle}. The boundary vertices of the quiver are denoted by $1$, $\ldots$, $6$. 
	
	\begin{figure}[!htb]
		\begin{center}
			\begin{tikzpicture}[scale=1]
			\node(a) at (-1,-1){};
			\node(b) at (2,-1){};
			\node(c) at (5,-1){};
			\node(d) at (0.5,1.12){};
			\node(e) at (2,3.24){};
			\node(f) at (3.5,1.12){};
			\node(ab) at (0.5, -1){};
			\node(bc) at (3.5, -1){};
			\node(ad) at (-0.25,0.06){};
			\node(de) at (1.25,2.18){};
			\node(ef) at (2.75,2.18){};
			\node(cf) at (4.25,0.06){};
			\node(abd) at (0.5,-0.3){};
			\node(bdf) at (2,0.4){};
			\node(bcf) at (3.5,-0.3){};
			\node(def) at (2,1.83){};
			
			\draw[fill=white] (abd) circle (3pt);
			\draw[fill=white] (def) circle (3pt);
			\draw[fill=white] (bcf) circle (3pt);
			\draw[fill] (ab) circle (2pt);
			\draw[fill] (bc) circle (2pt);
			\draw[fill] (ad) circle (2pt);
			\draw[fill] (de) circle (2pt);
			\draw[fill] (ef) circle (2pt);
			\draw[fill] (cf) circle (2pt);
			\draw[fill] (bdf) circle (2pt);

			\draw (ab) -- (abd);
			\draw (ad) -- (abd);
			\draw (de) -- (def);
			\draw (ef) -- (def);
			\draw (cf) -- (bcf);
			\draw (bc) -- (bcf);
			\draw (abd) -- (bdf);
			\draw (bcf) -- (bdf);
			\draw (def) -- (bdf);
			
			\draw[dotted] (a) to (b);
			\draw[dotted] (b) to (c);
			\draw[dotted] (a) to (d);
			\draw[dotted] (d) to (e);
			\draw[dotted] (e) to (f);
			\draw[dotted] (f) to (c);
			\draw[dotted] (a) to (b);
			\draw[dotted] (a) to (b);
			\draw[dotted] (a) to (b);
			\draw[dotted] (a) to (b);
			\draw[dotted] (a) to (b);
			\draw[dotted] (d) to (b);
			\draw[dotted] (d) to (f);
			\draw[dotted] (f) to (b);

			\node(1)[label=above: $\textcolor{red}1$] at (2,2.53){};
			\node(3)[label=right: $\textcolor{red}5$] at (4.25,-0.64){};
			\node(2)[label=left: $\textcolor{red}3$] at (-0.25,-0.64){};
			\node(23)[label=below: $\textcolor{red}4$] at (2,-0.64){};
			\node(13)[label=right: $\textcolor{red}6$] at (3.125,0.945){};
			\node(12)[label=left: $\textcolor{red}2$] at (0.875,0.945){};
			
			\path[red,line width=0.5mm,->] (1) edge (12);
			\path[red,line width=0.5mm,->] (12) edge (2);
			\path[red,line width=0.5mm,->] (2) edge (23);
			\path[red,line width=0.5mm,->] (23) edge (3);
			\path[red,line width=0.5mm,->] (3) edge (13);
			\path[red,line width=0.5mm,->] (13) edge (1);
			\path[red,line width=0.5mm,->] (23) edge (12);
			\path[red,line width=0.5mm,->] (12) edge (13);
			\path[red,line width=0.5mm,->] (13) edge (23);
			
			\draw[fill=red] (1) circle (2pt);
			\draw[fill=red] (2) circle (2pt);
			\draw[fill=red] (3) circle (2pt);
			\draw[fill=red] (12) circle (2pt);
			\draw[fill=red] (13) circle (2pt);
			\draw[fill=red] (23) circle (2pt);
			\end{tikzpicture}
			\vspace{-0.3 cm}
			\caption{Quiver of the GL$_2$-dimer of a triangle.\label{fig:quiver_triangle}}
		\end{center}
	\end{figure}
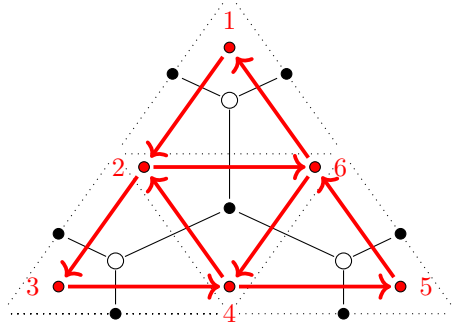 
	\vspace{-0.8 cm}
	\begin{defn}[chordless cycle]
		A chordless cycle of a quiver $Q$ is a cycle such that the full subquiver on its vertices is also a cycle. 
	\end{defn}
	
	\begin{remark}\label{prop:equivalent_cycle}
		The following fact is due to the definition of a dimer model with boundary:\\ 
		Let $Q$ be the dimer model of a GL$_m$-dimer and $\Lambda_Q$ the corresponding dimer algebra. Let $k$ be an arbitrary vertex of $\Lambda_Q$  with at least two incoming and two outgoing arrows. Then, up to $\partial W$, $c_1=c_2$ for any two chordless cycles $c_1$, $c_2$ starting at $k$.
	\end{remark}
	
	Examples of quivers of GL$_m$-dimers are shown in Figure $\ref{fig:quiver_triangle}$ for $m=2$  and in Figure \ref{fig:quiver_gl_5_quadrilateral} for $m=5$. In our particular setting, the number of faces (cycles) incident with a vertex $i$ are always $1$ or $3$ for boundary vertices and $4$ or $6$ for internal vertices.
	
	By Remark \ref{prop:equivalent_cycle}, all chordless cycles at a given vertex are equal and hence it makes sense to refer to any one of them as the cycle at this vertex.
	\begin{defn}[short cycle $u$]\label{def:short_cycle} 
		Let $i$ be a vertex of $Q $, then we write $u_i$ for a chordless cycle at $i$.
	\end{defn}


	\section{The boundary algebras of dimer models of GL$_2$-dimers of arbitrary triangulations of the $n$-gon are isomorphic} 
	
	Recall that $\mathcal{B}_Q= e_b \Lambda_Q e_b$ is the boundary algebra obtained from the quiver $Q$ of a GL$_2$-dimer of a triangulation of an $n$-gon, where $e_b=e_1+ \ldots + e_{2n}$ denotes the sum of all boundary idempotents.\\ We also recall the quiver $\Gamma(n)$ from the introduction. For $n=5$, it has the following form $\Gamma(5)$:
	\begin{figure}[!htb]
		\begin{center}
			\begin{tikzpicture}[scale=1.5]
			\node(a) at (5,5){};
			\node(b) at (2.78,3.36){};
			\node(c) at (3.65371, 0.74187){};
			\node(d) at (6.4137, 0.76377){};
			\node(e) at (7.24575, 3.39544){};
			\node(ab) at (3.89, 4.18){};
			\node(bc) at (3.21686, 2.05093){};
			\node(cd) at (5.03371, 0.75282){};
			\node(de) at (6.82973, 2.07961){};
			\node(ae) at (6.12288, 4.19772){};
			\node(aab) at (4.445, 4.59){};
			\node(abb) at (3.335, 3.77){};
			\node(bbc) at (2.99843, 2.70547){};
			\node(bcc) at (3.43529, 1.3964){};
			\node(ccd) at (4.34371, 0.74734){};
			\node(cdd) at (5.7237, 0.7583){};
			\node(dde) at (6.62171, 1.42169){};
			\node(dee) at (7.03774, 2.73752){};
			\node(aee) at (6.68431, 3.79658){};
			\node(aae) at (5.56144, 4.59886){};
			\node(u11) at (3.24567, 3.24522){};
			\node(u12) at (3.6063, 1.90788){};
			\node(u13) at (4.29751, 3.9214){};
			\node(d1) at (3.86175, 3.01982){};
			\node(u21) at (4.89856, 3.53072){};
			\node(u22) at (4.31253, 1.45709){};
			\node(u23) at (5.63485, 1.47212){};
			\node(d2) at (4.92861, 2.28354){};
			\node(u31) at (6.68669, 3.36543){};
			\node(u32) at (5.6649, 3.93643){};
			\node(u33) at (6.43124, 1.9229){};
			\node(d3) at (6.16077, 3.10998){};
			\node(m1) at (4.32686, 2.87093){};
			\node(m2) at (5.70685, 2.88189){};
			\node(aac) at (4.66343, 3.93547){};
			\node(acc) at (3.99029, 1.8064){};
			\node(aad) at (5.35343, 3.94094){};
			\node(add) at (6.06028, 1.82283){};

			\node(1)[label=above: $\textcolor{red}1$] at (5.00525, 4.65015){};
			\node(2)[label=above: $\textcolor{red}2$] at (3.91357, 3.90125){};
			\node(3)[label=left: $\textcolor{red}3$] at (3.0064, 3.3431){};
			\node(4)[label=left: $\textcolor{red}4$] at (3.36004, 2.29755){};
			\node(5)[label=below: $\textcolor{red}5$] at (3.79056, 0.95986){};
			\node(6)[label=below: $\textcolor{red}6$] at (5.02062, 0.95986){};
			
			\node(7)[label=below: $\textcolor{red}7$] at (6.2738, 0.95986){};
			\node(8)[label=right: $\textcolor{red}8$] at (6.63268, 2.2204){};
			\node(9)[label=right: $\textcolor{red}9$] at (6.98872, 3.32772){};
			\node(10)[label=above: $\textcolor{red}{10}$] at (6.15843, 3.89662){};
			
			\path[red,line width=0.5mm,->] (1) edge (2);
			\node(x2)[label=above: $\textcolor{red}{x_2}$] at (4.4,4.2){};
			\path[red,line width=0.5mm,->] (2) edge (3);
			\node(x3)[label=above: $\textcolor{red}{x_3}$] at (3.4,3.6){};
			\path[red,line width=0.5mm,->] (3) edge (4);
			\node(x4)[label=above: $\textcolor{red}{x_4}$] at (3.02,2.65){};
			\path[red,line width=0.5mm,->] (4) edge (5);
			\node(x5)[label=above: $\textcolor{red}{x_5}$] at (3.395,1.43){};
			\path[red,line width=0.5mm,->] (5) edge (6);
			\node(x6)[label=above: $\textcolor{red}{x_6}$] at (4.4,0.6){};
			\path[red,line width=0.5mm,->] (6) edge (7);
			\node(x7)[label=above: $\textcolor{red}{x_7}$] at (5.6,0.6){};
			\path[red,line width=0.5mm,->] (7) edge (8);
			\node(x8)[label=above: $\textcolor{red}{x_8}$] at (6.65,1.43){};
			\path[red,line width=0.5mm,->] (8) edge (9);
			\node(x9)[label=above: $\textcolor{red}{x_9}$] at (7.02,2.65){};
			\path[red,line width=0.5mm,->] (9) edge (10);
			\node(x10)[label=above: $\textcolor{red}{x_{10}}$] at (6.6,3.6){};
			\path[red,line width=0.5mm,->] (10) edge (1);
			\node(x1)[label=above: $\textcolor{red}{x_1}$] at (5.6,4.2){};

			\path[red,line width=0.5mm,->] (4) edge (2);
			\node(z4)[label=above: $\textcolor{red}{z_4}$] at (3.8,2.9){};
			\path[red,line width=0.5mm,->] (2) edge (10);
			\node(z2)[label=above: $\textcolor{red}{z_2}$] at (5,3.5){};
			\path[red,line width=0.5mm,->] (10) edge (8);
			\node(z10)[label=above: $\textcolor{red}{z_{10}}$] at (6.2,2.9){};
			\path[red,line width=0.5mm,->] (8) edge (6);
			\node(z8)[label=above: $\textcolor{red}{z_8}$] at (5.8,1.6){};
			\path[red,line width=0.5mm,->] (6) edge (4);
			\node(z6)[label=above: $\textcolor{red}{z_6}$] at (4.2,1.6){};

			\draw[fill=red] (1) circle (1pt);
			\draw[fill=red] (2) circle (1pt);
			\draw[fill=red] (3) circle (1pt);
			\draw[fill=red] (4) circle (1pt);
			\draw[fill=red] (5) circle (1pt);
			\draw[fill=red] (6) circle (1pt);
			\draw[fill=red] (7) circle (1pt);
			\draw[fill=red] (8) circle (1pt);
			\draw[fill=red] (9) circle (1pt);
			\draw[fill=red] (10) circle (1pt);
			\end{tikzpicture}
			\vspace{-0.3 cm}
			\caption{$\Gamma(5)$.\label{gamma_5}}
		\end{center}
	\end{figure} 
	\vspace{-0.8 cm}
	\begin{theorem}[Main Theorem]\label{th:main_theorem}
		
		The quiver of $\mathcal{B}_Q$ with relations $\partial W$ is isomorphic to $\Gamma(n)$ subject to the following relations (writing compositions of paths from left to right), for $i=1,\ldots,n$, where indices are considered modulo $2n$:
		\begin{eqnarray*}
			x_{2i+1}x_{2i+2} z_{2i+2} & = &  z_{2i}x_{2i-1}x_{2i} \\
			z_{2i}z_{2i-2} & = & x_{2i+1}x_{2i+2} \ldots x_{2i+2\cdot(n-2)}.
		\end{eqnarray*}
		\\ \
		Furthermore the element
		\begin{align*}
		t := \sum_{i=1}^{n}{x_{2i-1}x_{2i}z_{2i}}+ \sum_{i=1}^{n}{x_{2i}z_{2i}x_{2i-1}}
		\end{align*}
		is central in $\mathcal{B}_Q$.
	\end{theorem}
	
	The proof of Theorem \ref{th:main_theorem} is split into two main steps. First, we consider fan triangulations and show by induction, that in this case the boundary algebra $\mathcal{B}$ has the desired structure for any $n$. The second step is to show that the flip of a diagonal does not change the boundary algebra i.e. $\mathcal{B}$ is flip-invariant. Using the fact that every triangulation of a polygon
	can be reached from any starting triangulation under application of finitely many
	flips (Theorem (a) in \cite{h}), we get the claimed result.


	\subsection{The boundary algebra of a fan triangulation}
	The goal of this chapter is to show that the boundary algebra of a  fan triangulation is isomorphic to the algebra $\mathcal{B}$. Before describing the boundary algebra, the structure of the quiver of a fan triangulation of an $n$-gon will be determined. We will write $Q_F(n)$ for the dimer model of the GL$_2$-dimer of a fan triangulation of the $n$-gon.
	
	\begin{figure}[!htb]
		\begin{center}
			\begin{tikzpicture}[scale=0.5]
			\node(1)[label=above: $\textcolor{red}1$] at (0,15.08202){};
			\node(2)[label=above: $\textcolor{red}2$] at (-2.77164,13.93397){};
			\node(3)[label=above: $\textcolor{red}3$] at (-4.89296, 11.81265){};
			\node(4)[label=left: $\textcolor{red}4$] at (-6.04101, 9.04101){};
			\node(5)[label=left: $\textcolor{red}5$] at (-6.04101, 6.04101){};
			\node(6)[label=below: $\textcolor{red}6$] at (-4.89296, 3.26937){};
			\node(7)[label=below: $\textcolor{red}7$] at (-2.77164, 1.14805){};
			\node(8)[label=below: $\textcolor{red}8$] at (0,0){};
			\node(9)[label=below: $\textcolor{red}9$] at (3,0){};
			\node(10)[label=below: $\textcolor{red}{10}$] at (5.77164, 1.14805){};
			\node(11)[label=right: $\textcolor{red}{11}$] at (7.89296, 3.26937){};
			\node(12)[label=right: $\textcolor{red}{12}$] at (9.04101, 6.04101){};
			\node(13)[label=right: $\textcolor{red}{13}$] at (9.04101, 9.04101){};
			\node(14)[label=above: $\textcolor{red}{14}$] at (7.89296, 11.81265){};
			\node(15)[label=above: $\textcolor{red}{15}$] at (5.77164, 13.93397){};
			\node(16)[label=above: $\textcolor{red}{16}$] at (3, 15.08202){};
			\node(i1)[label=right: $\textcolor{red}{i_1}$] at (-3.54703, 9.81999){};
			\node(i2) at (-1.86093, 6.54144){};
			\node(i2_label)[label=above: $\textcolor{red}{i_2}$] at (-1.66093, 6.34144){};
			\node(i3)[label=above: $\textcolor{red}{i_3}$] at (2, 6){};
			\node(i4) at (5.30503, 8.18071){};
			\node(i4_label)[label=left: $\textcolor{red}{i_4}$] at (5.45503, 8.28071){};
			\node(i5)[label=left: $\textcolor{red}{i_5}$] at (5.35186, 11.22507){};
			
			\draw[fill=red] (1) circle (1pt);
			\draw[fill=red] (2) circle (1pt);
			\draw[fill=red] (3) circle (1pt);
			\draw[fill=red] (4) circle (1pt);
			\draw[fill=red] (5) circle (1pt);
			\draw[fill=red] (6) circle (1pt);
			\draw[fill=red] (7) circle (1pt);
			\draw[fill=red] (8) circle (1pt);
			\draw[fill=red] (9) circle (1pt);
			\draw[fill=red] (10) circle (1pt);
			\draw[fill=red] (11) circle (1pt);
			\draw[fill=red] (12) circle (1pt);
			\draw[fill=red] (13) circle (1pt);
			\draw[fill=red] (14) circle (1pt);
			\draw[fill=red] (15) circle (1pt);
			\draw[fill=red] (16) circle (1pt);
			\draw[fill=red] (i1) circle (1pt);
			\draw[fill=red] (i2) circle (1pt);
			\draw[fill=red] (i3) circle (1pt);
			\draw[fill=red] (i4) circle (1pt);
			\draw[fill=red] (i5) circle (1pt);
			\node(x1)[label=above: $\textcolor{red}{x_1}$] at (1.5,14.88){}; 
			
			\path[red, line width=0.5mm,->] (16) edge  (1) ;
			\path[red, line width=0.5mm,->] (1) edge  (2) ;
			\path[red, line width=0.5mm,->] (2) edge  (3) ;
			\path[red, line width=0.5mm,->] (3) edge  (4) ;
			\path[red, line width=0.5mm,->] (4) edge  (5) ;
			\path[red, line width=0.5mm,->] (5) edge  (6) ;
			\path[red, line width=0.5mm,->] (6) edge  (7) ;
			\path[red, line width=0.5mm,->] (7) edge  (8) ;
			\path[red, line width=0.5mm,->] (8) edge  (9) ;
			\path[red, line width=0.5mm,->] (9) edge  (10) ;
			\path[red, line width=0.5mm,->] (10) edge  (11) ;
			\path[red, line width=0.5mm,->] (11) edge  (12) ;
			\path[red, line width=0.5mm,->] (12) edge  (13) ;
			\path[red, line width=0.5mm,->] (13) edge  (14) ;
			\path[red, line width=0.5mm,->] (14) edge  (15) ;
			\path[red, line width=0.5mm,->, shorten >=0.2cm] (15) edge  (16) ;
			
			\path[red, line width=0.5mm,->, shorten >=0.2cm] (16) edge (14)  ;
			\node(y16)[label=below: $\textcolor{red}{y_{16}}$] at (5.4,13.5){}; 
			
			\path[red, line width=0.5mm,->, shorten >=0.15cm] (i5) edge  (16) ;
			\node(a5)[label=below: $\textcolor{red}{\alpha_{5}}$] at (3.9,13.4){};
			\path[red, line width=0.5mm,->, shorten >=0.15cm] (14) edge  (i5) ;
			\node(c5)[label=below: $\textcolor{red}{\gamma_{5}}$] at (6.9,11.7){};
			\path[red, line width=0.5mm,->, shorten >=0.15cm] (i5) edge  (12) ;
			\node(b4)[label=below: $\textcolor{red}{\beta_{4}}$] at (7.3,10.2){};
			
			\path[red, line width=0.5mm,->, shorten >=0.15cm] (i4) edge  (i5) ;
			\node(a4)[label=below: $\textcolor{red}{\alpha_{4}}$] at (4.85,10.0){};
			\path[red, line width=0.5mm,->, shorten >=0.15cm] (12) edge  (i4) ;
			\node(c4)[label=below: $\textcolor{red}{\gamma_{4}}$] at (6.9,7.3){};
			\path[red, line width=0.5mm,->, shorten >=0.15cm] (i4) edge  (10) ;
			\node(b3)[label=below: $\textcolor{red}{\beta_{3}}$] at (6.0,5.4){};

			\path[red, line width=0.5mm,->, shorten >=0.15cm] (i3) edge  (i4) ;
			\node(a3)[label=below: $\textcolor{red}{\alpha_{3}}$] at (3.3,8.1){};
			\path[red, line width=0.5mm,->, shorten >=0.15cm] (10) edge  (i3) ;
			\node(c3)[label=below: $\textcolor{red}{\gamma_{3}}$] at (3.9,4.8){};
			\path[red, line width=0.5mm,->, shorten >=0.15cm] (i3) edge  (8) ;
			\node(b2)[label=below: $\textcolor{red}{\beta_{2}}$] at (0.55,4.0){};

			\path[red, line width=0.5mm,->, shorten >=0.15cm] (i2) edge  (i3) ;
			\node(a2)[label=below: $\textcolor{red}{\alpha_{2}}$] at (0.2,7.3){};
			\path[red, line width=0.5mm,->, shorten >=0.15cm] (8) edge  (i2) ;
			\node(c2)[label=below: $\textcolor{red}{\gamma_{2}}$] at (-0.6,4.0){};
			\path[red, line width=0.5mm,->, shorten >=0.15cm] (i2) edge  (6) ;
			\node(b1)[label=below: $\textcolor{red}{\beta_{1}}$] at (-3.0,5.4){};
			
			\path[red, line width=0.5mm,->, shorten >=0.15cm] (i1) edge  (i2) ;
			\node(a1)[label=below: $\textcolor{red}{\alpha_{1}}$] at (-2.15,8.9){};
			\path[red, line width=0.5mm,->, shorten >=0.15cm] (6) edge  (i1) ;
			\node(c1)[label=below: $\textcolor{red}{\gamma_{1}}$] at (-3.9,7.0){};
			\path[red, line width=0.5mm,->, shorten >=0.15cm] (i1) edge  (4) ;
			\node(b0)[label=below: $\textcolor{red}{\beta_{0}}$] at (-4.9,9.5){};
			
			\path[red, line width=0.5mm,->, shorten >=0.15cm] (2) edge  (i1) ;
			\node(a0)[label=below: $\textcolor{red}{\alpha_{0}}$] at (-2.7,12.0){};
			\path[red, line width=0.5mm,->, shorten >=0.15cm] (4) edge  (2) ;
			\node(y4)[label=below: $\textcolor{red}{y_{4}}$] at (-4.0,12.0){};
			\end{tikzpicture}
			\vspace{-0.3 cm}
			\caption{$Q_F(8)$: The quiver of a fan triangulation of the octagon.\label{labelled_quiver_octagon}}
		\end{center}
	\end{figure}
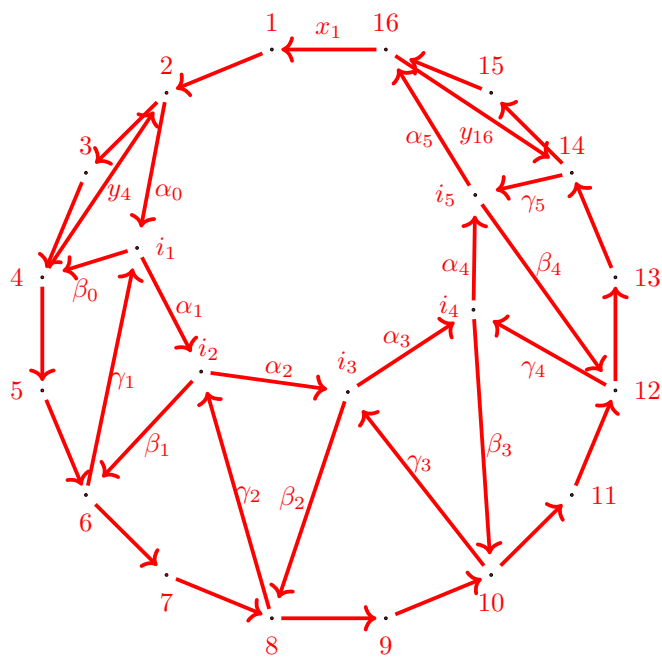 
	\begin{prop}\label{prop:quiver_fan_triangulation}
		Let $Q_F(n)$ be the  dimer model with boundary of a fan triangulation of the $n$-gon, $n\geq3$. Then $Q_F(n)$ has the following form: \\
		It consists of $2n$ vertices on the boundary, labelled anticlockwise by $1 , \ldots, 2n$, and $n-3$ internal vertices labelled $i_1 , \ldots, i_{n-3}$. \\
		Furthermore it has $2n+2$ arrows between the boundary vertices, with $k\in [1,2n]$ and indices taken modulo $2n$,
		\begin{align*}
		x_k &: k-1 \rightarrow k  \\
		y_4 &: 4\rightarrow 2 \\
		y_{2n} &: 2n\rightarrow 2n-2,
		\end{align*}
		and arrows
		\begin{align*}
		\alpha_0 &:2 \rightarrow i_1 \ \ & \ \\
		\alpha_k &: i_k \rightarrow i_{k+1}  \ \ &1 \leq k < n-3 \\
		\alpha_{n-3}&: i_{n-3} \rightarrow 2n & \ \\
		\beta_{k-1}&: i_k\rightarrow 2k +2 \ \ &1 \leq k \leq n-3 \\
		\gamma_k &: 2k+4 \rightarrow i_k \ \ &1 \leq k \leq n-3. 
		\end{align*}
	\end{prop} 
	\begin{rem}
		The quiver $Q_F(n)$ has $2n-2$ faces.
	\end{rem}
	The quiver $Q_F(8)$ is illustrated in Figure \ref{labelled_quiver_octagon}.
	\begin{proof}
		The claim follows inductively by removing  $2$-cycles of the dimer model using the relations obtained by the natural potential $W$. Note that on the dimer, this reduction corresponds to replacing every subgraph of the form shown on left hand side in Figure \ref{fig:replacing_dimer} to the form shown on the right hand side.
		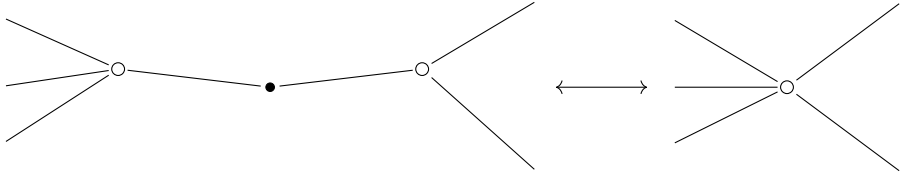
\begin{figure}[!htb]
			\begin{center}
				\begin{tikzpicture}[scale=0.8]
				\node(a) at (1.5,0){};
				\node(b) at (-1,0.3){};
				\node(c) at (4,0.3){};
				\node(b1) at (-3,1.2){};
				\node(b2) at (-3,0){};
				\node(b3) at (-3,-1.0){};
				\node(c1) at (6,1.5){};
				\node(c2) at (6,-1.5){};
				
				\draw[<->] (6.2,0) -- (7.7,0);
				
				\node(d) at (10,0){};
				\node(d1) at (8,1.2){};
				\node(d2) at (8,0){};
				\node(d3) at (8,-1.0){};
				\node(d4) at (12,1.5){};
				\node(d5) at (12,-1.5){};
				\draw[fill=white] (b) circle (3pt);
				\draw[fill=white] (c) circle (3pt);
				\draw[fill=white] (d) circle (3pt);
				\draw[fill] (a) circle (2pt);
				\draw (a) -- (b);
				\draw (c) -- (a);
				\draw (b) -- (b1);
				\draw (b) -- (b2);
				\draw (b) -- (b3);
				\draw (c) -- (c1);
				\draw (c) -- (c2);
				\draw (d) -- (d1);
				\draw (d) -- (d2);
				\draw (d) -- (d3);
				\draw (d) -- (d4);
				\draw (d) -- (d5);
				\end{tikzpicture}
				\vspace{-0.3 cm}
				\caption{Reduction procedure.\label{fig:replacing_dimer}}
			\end{center}
		\end{figure} 
		The case $n=3$ is the induction basis. The quiver of $Q_F(3)$ has the claimed form, see Figure \ref{fig:quiver_triangle}.\\
		Assume that $Q_F(n)$ has the described form for $n\geq 3$ and consider the GL$_2$-dimer of the fan triangulation of the $n\!+\!1$-gon. It can be obtained from the fan triangulation  of the $n$-gon by adding a triangle to it at vertex $1$ to the end of the fan. The dimer is obtained analogously.
		Observe that the same reduction steps can be done for the GL$_2$-dimer of the $n\!+\!1$-gon as for the one of the $n$-gon, because the only difference between the two triangulations is the  additional triangle between $1$, $n$ and $n\!+\!1$ , which does not change the GL$_2$-dimer of the former $n$-gon. The reduction steps are the one described in Figure \ref{fig:replacing_dimer}.
		\begin{figure}[!htb]
			\begin{center}
				\includegraphics[scale=0.28]{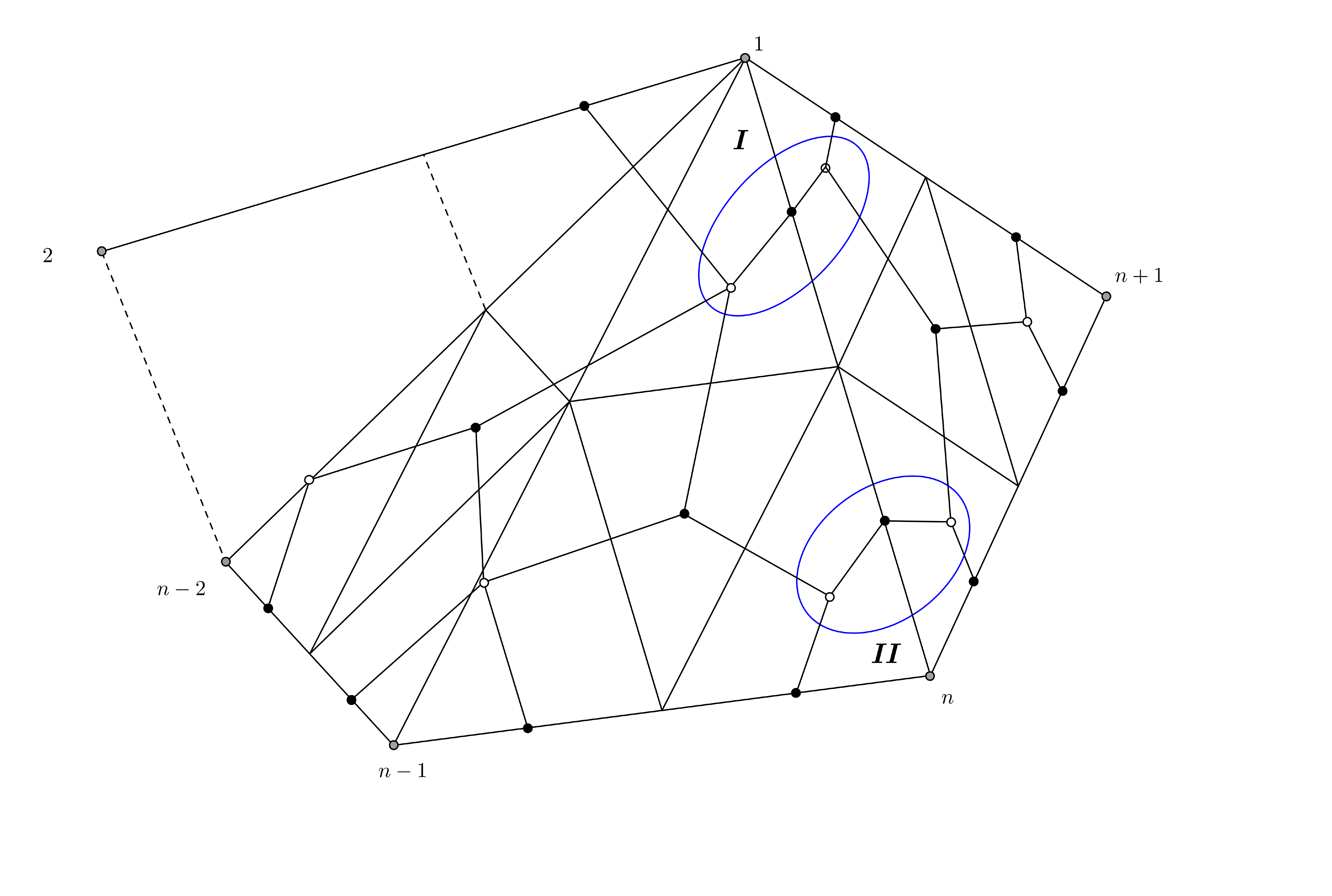}
				\vspace{-2.0 cm}
				\caption{A fan triangulation of the $n\!+\!1$-gon and new part of the dimer. The labelling corresponds to the vertices of the $n\!+\!1$-gon. \label{fig:quiver_general_reducing_new}}
			\end{center}
		\end{figure} 
		Figure \ref{fig:quiver_general_reducing_new} shows the relevant part of the reduced dimer, i.e. the new part obtained by increasing the number of vertices of the polygon. Note that it is possible to reduce the new dimer as the regions $I$ and $II$ indicate. This leads to the reduced dimer shown in Figure \ref{fig:quiver_general_reduced_new}, containing the reduced quiver $Q_F(n+1)$. The new quiver has $2$ additional faces, the chordless cycles $2n$,$2n+1$,$2n+2$ and $2n$,$i_{n-2}$,$2n+2$. So it has the claimed structure. 
		
		\begin{figure}[!htb]
			\begin{center}
				\includegraphics[scale=0.28]{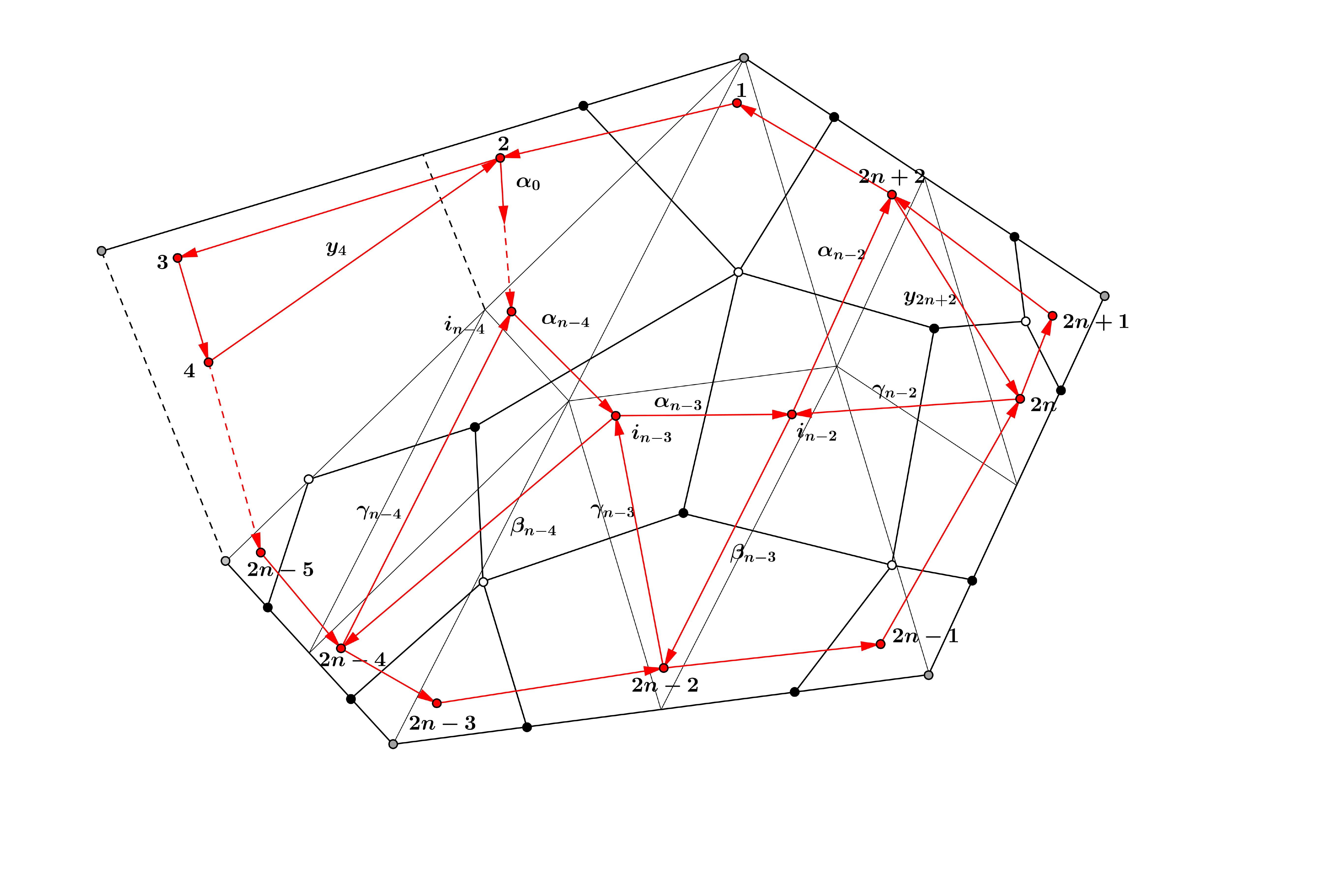}
				\vspace{-2.3 cm}
				\caption{Part of a reduced GL$_2$-dimer of the $n\!+\!1$-gon and the quiver $Q(n+1)$.\label{fig:quiver_general_reduced_new}}
			\end{center}
		\end{figure} 
	\end{proof}
	\begin{rem}
		Whenever we have a $2$-cycle of internal arrows in the quiver of a GL$_m$-dimer, we can remove it using the relations from the potential. From now on we will always tacitly remove such $2$-cycles from our dimer model and will use the phrases GL$_m$-dimer and quiver instead of "reduced GL$_m$-dimer" and "reduced quiver".	
	\end{rem}
	Knowing the structure of the quiver of a fan triangulation in detail, it is now  possible to describe the boundary algebra of the $n$-gon.
	
	\begin{defn}\label{def:z_2k}
		We define paths $z_{2},\ldots,z_{2n}$ as follows:
		\begin{align}
		z_{2k} &:=  \gamma_{k-2}\beta_{k-3} \text{ for $k=3,\ldots,n-1$} \label{eq:general_boundary_algebra_1} \\
		z_4 &:= y_4   \label{eq:general_boundary_algebra_2}\\
		z_2 &:= \alpha_0 \alpha_1 \ldots \alpha_{n-3}  \label{eq:general_boundary_algebra_3}\\
		z_{2n} &:= y_{2n}.  \label{eq:general_boundary_algebra_4}
		\end{align}
	\end{defn}
	\begin{lm}\label{lemma:generator_fan}
		The $x_i$, $i=1,\ldots,2n$ together with the $z_{2k}$, $k=1,\ldots,n$ as defined in Definition \ref{def:z_2k} generate the boundary algebra of $Q_F(n)$.
	\end{lm}
	\begin{rem}
		The set $\left \{ \{x_i\}_{1\leq i \leq 2n}, \{z_{2j} \}_{1\leq j \leq n} \right \}$ is minimal in the sense that any proper subset does not suffice.
	\end{rem}
	\begin{proof}
		For the convenience of the reader, Figure \ref{labelled_quiver_octagon} shows $Q_F(8)$ in order to make it easier to follow the argumentation by using relations obtained by the natural potential $W$ for different internal arrows.\\
		We show that each path from $2$ to $2n$ factors through $z_2$. Assume to the contrary that a path $\delta$ from $2$ to $n$ does not factor through $z_2$. We can assume that $\delta$ does not contain cycles. Then the following two cases can occur: Either $\delta=x_3x_4\cdots x_{2n}$ or there exists a $k$, $1\leq k \leq n-3$ such that $\alpha_k$ is not an arrow of $\delta$, w.l.o.g let $k$ be minimal, i.e. $\delta= \alpha_0 \ldots \alpha_{k-1} \widetilde{\delta}$. In the first case we get the equivalence of the following paths: 
		\begin{align}\label{calculation:product_x}
		x_3x_4\cdots x_{2n} \overset{y_4}{\cong}\alpha_0 \beta_0 x_5x_6 \cdots x_{2n} \overset{\gamma_1}{\cong} \ldots \overset{\gamma_{n-3}}{\cong} \alpha_0 \cdots \alpha_{n-3} y_{2n} x_{2n-1}x_{2n}=z_{2} u_{2n}.  
		\end{align}
		The last equality holds as the path $y_{2n}x_{2n-1}x_{2n}$ is a chordless cycle at $2n$. Hence $\delta$ factors through $z_{2}$ in the first case.
		In the second case, as $\delta$ does not contain cycles, $\widetilde{\delta}$ must be of the form $\beta_{k-1} x_{2k+1}x_{2k+2}\widetilde{\delta}'$ for some path $\widetilde{\delta}'$. Then
		\begin{align*}
		\widetilde{\delta}= \beta_{k-1} x_{2k+1}x_{2k+2}\widetilde{\delta}' \overset{\gamma_k}{\cong} \alpha_k \beta_k \widetilde{\delta}'
		\end{align*}
		which is a contradiction to the minimality of $k$. Hence every path from $2$ to $2n$ factors through $z_{2}$. The rest of the statement follows with similar arguments.
	\end{proof}
	
	Now we want to prove that the relations between the arrows, stated in the previous section are fulfilled for the boundary algebra of the fan triangulation of a polygon. Let $\Lambda_{Q_F(n)}$ be the dimer algebra of $Q_F(n)$ and let $e_b = \sum_{k=1}^{2n}{e_k}$. Note that we will always reduce indices modulo $2n$.
	
	\begin{prop}\label{prop:relation}
		The boundary algebra $e_b \Lambda_{Q_F(n)} e_b$ satisfies the following relations for $k\in[1,n]$:
		\begin{itemize}
			\item[(I.)] $z_{2k}z_{2k-2}  =  x_{2k+1}x_{2k+2} \ldots x_{2k+2\cdot(n-2)}$
			\item[(II.)]  $x_{2k+1} x_{2k+2} z_{2k+2} = z_{2k} x_{2k-1} x_{2k}$.
		\end{itemize}
	\end{prop}
	\begin{proof}
		By the same calculations as in (\ref{calculation:product_x}) of Lemma \ref{lemma:generator_fan}  we immediately get the equality of the paths 
		\begin{align*}
		x_3 x_4 \cdots x_{2n-3}x_{2n-2} = z_{2}y_{2n} = z_2 z_{2n}.
		\end{align*}
		All other relations of type (I.) follow from Lemma \ref{lemma:generator_fan} in the same way.
		The second kind of relation has already been stated in Remark \ref{prop:equivalent_cycle}, as both sides of the equation are short cycles $u_{2k}$ at boundary vertex $2k$.
	\end{proof}
	Thus we described the boundary algebra of the GL$_2$-dimer of a fan triangulation for arbitrary large $n$ in detail and it remains to show flip invariance in order to get the main result for $m=2$.


	\subsection{Flips in the boundary algebras}\label{subsec:flips}
	
	This section starts with the definition of a diagonal flip of a triangulation in order to show that a flip does not change the structure of the boundary algebra itself. As already shown the boundary algebra of the fan triangulation has the structure  given in Theorem \ref{th:main_theorem}. Together with the main result of this section (Theorem \ref{th:flip}), this proves that all boundary algebras arising from GL$_2$-dimers of arbitrary triangulations of an $n$-gon are isomorphic.
	\begin{defn}[Diagonal flip of a triangulation.]
		For a triangulation a diagonal flip is defined as follows. Let $(l,j)$ be a diagonal of the triangulation of the $n$-gon. Then two triangles $l$,$j$,$k$ and $l$,$j$,$i$ belong to the triangulation. A flip, as shown in Figure \ref{fig:diagonal_flip},
		\vspace{-0.6 cm}
		\begin{center}
			\includegraphics[scale=0.8]{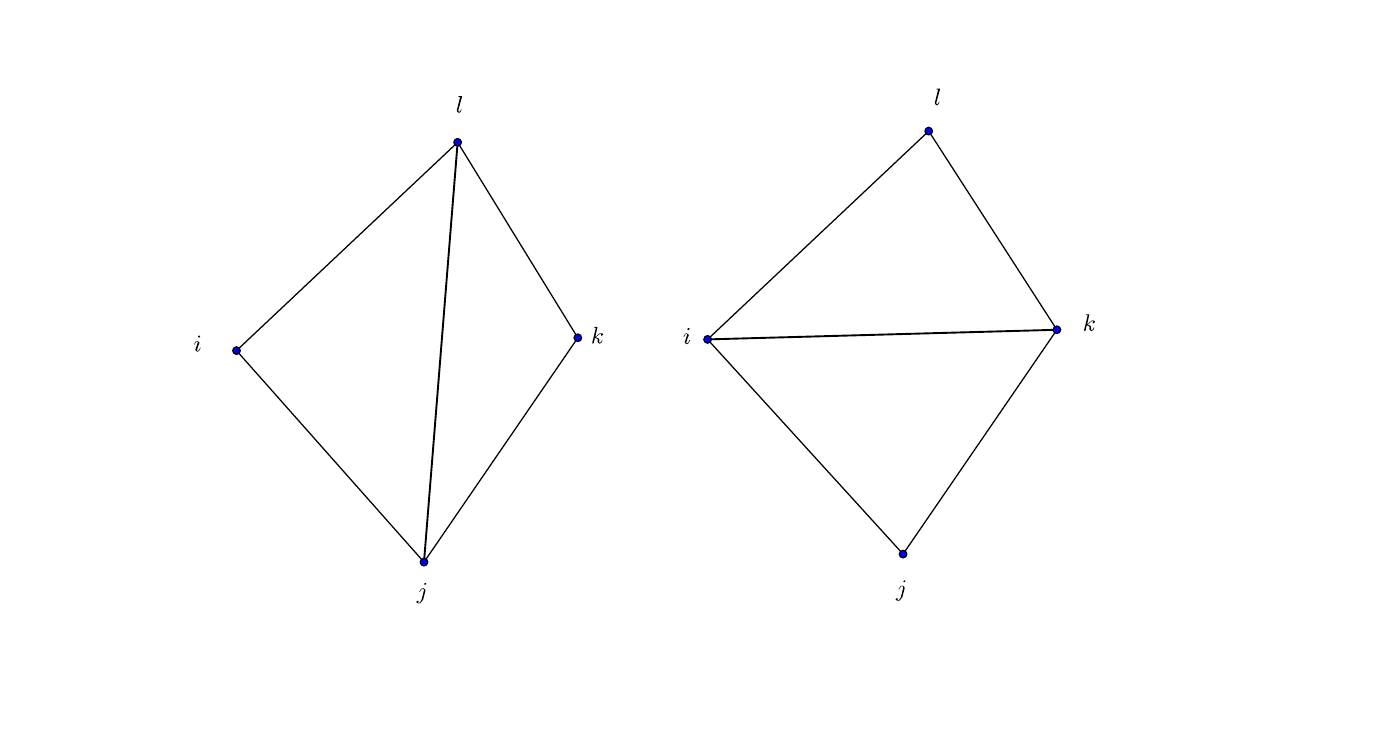}
			\vspace{-1.3 cm}
			\captionof{figure}{Diagonal flip of a triangulation.\label{fig:diagonal_flip}}
		\end{center}
		is the removal of the diagonal $(l,j)$ replacing it by the diagonal $(i,k)$. 
	\end{defn}
	
	Note that a diagonal flip is always a local operation that only changes the structure around vertices corresponding to the edges of the quadrilateral, and hence a local operation on the GL$_2$-dimer and the corresponding dimer algebra. Furthermore, every  triangulation of a polygon can be reached from any starting triangulation by application of finitely many flips, see Theorem (a) in \cite{h}. Recall that $Q_{F(n)}$ denotes the quiver of the GL$_2$-dimer of a fan triangulation of an $n$-gon with dimer algebra $\Lambda_{Q_{F(n)}}$ and $e_b$ the sum of the boundary idempotents of $\Lambda_{Q_{F(n)}}$.
	
	\begin{lm}\label{lm:flip_equivalence}
		Let $j\in [3,n-1]$ and $\mu$ be the flip of the diagonal $(1,j)$ of $Q_F(n)$. Let $e_b'$ be the sum of the boundary idempotents of $\Lambda_{\mu Q_F(n)}$. Then there is an isomorphism 
		\begin{align*}
		e_{b'} \Lambda_{\mu Q_F(n)} e_{b'} \cong e_b \Lambda_{Q_F(n)} e_b.  
		\end{align*}
	\end{lm}
	\begin{proof}
		By Lemma \ref{lemma:generator_fan} the  $x_i$, $i=1,\ldots,2n$ together with the $z_{2k}$, $k=1,\ldots,n$ generate the boundary algebra $\mathcal{B}(Q_F(n))$ and these elements fulfil relations (I.) and (II.) by Proposition \ref{prop:relation}.\\
		We consider the flip $\mu$ of the diagonal $(1,j)$ of the $n$-gon and the new dimer algebra $\Lambda_{\mu Q(F)}$, the relevant part is shown in Figure \ref{fig:general_flip_fan}.
		\begin{center}
			\includegraphics[scale=0.85]{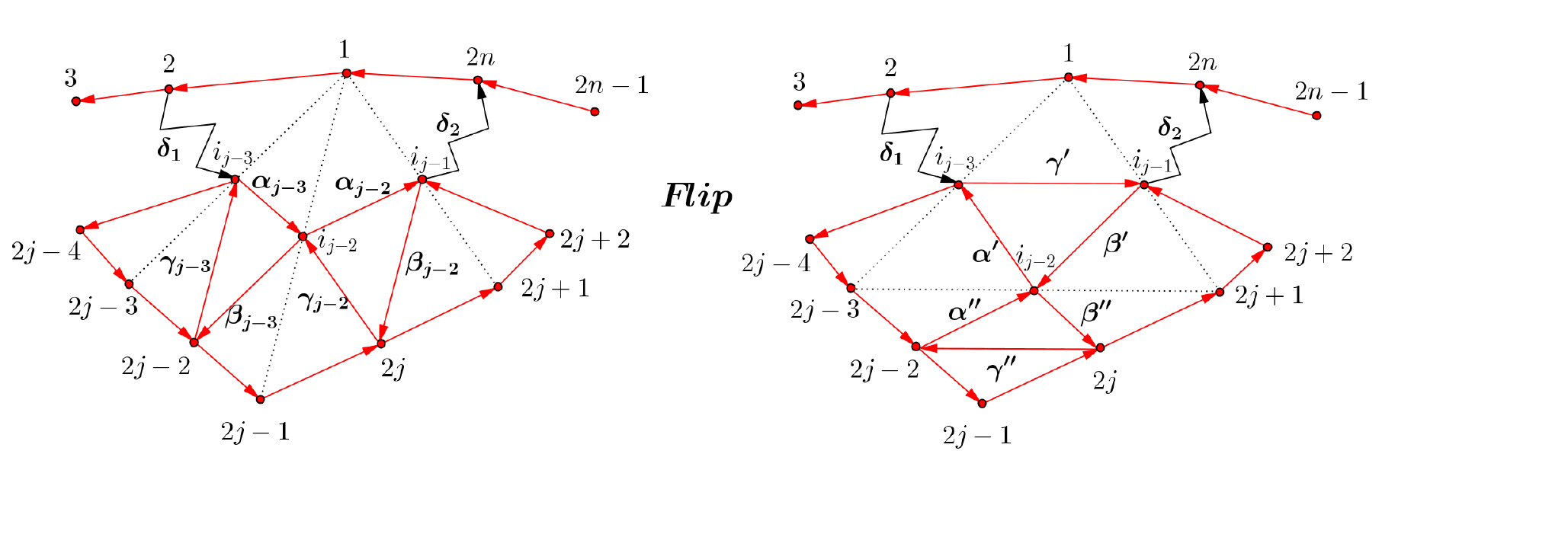}
			\vspace{-1.5 cm}
			\captionof{figure}{Flip of the diagonal $(1,j)$  changes the dimer algebra to $\Lambda_{\mu Q_F(n)}$ .\label{fig:general_flip_fan}}
		\end{center}
		Here the paths $\delta_1$ and $\delta_2$ are
		\begin{align*}
		\delta_1 &= \left\{ \begin{matrix} \alpha_0 \cdots \alpha_{j-4} && 4\leq j\leq n-1 \\
		e_2 						&& j=3 \end{matrix} \right. \\
		\delta_2 &= \left\{ \begin{matrix} \alpha_{j-1} \cdots \alpha_{n-3} && 3\leq j\leq n-2 \\
		e_{2n} 						  && j=n-1 \end{matrix} \right.
		\end{align*}
		and hence might be empty. Furthermore, if $j=3$, then $\gamma_{j-3}=y_4$, $\beta_{j-4}=e_2$, $i_{j-3}=2$ and if $j=n-1$, then $\gamma_{j-1}=e_{2n}$, $\beta_{j-2}=y_{2n}$ and $i_{j-1}=2n$.\\ 
		We know that every path of $\mathcal{B}(Q_F(n))$ from $2$ to $2n$ factors through $z_2=\delta_1 \alpha_{j-3} \alpha_{j-2} \delta_2$.
		The boundary algebra $e_{b'} \Lambda_{Q_F(n)} e_{b'}=:\mathcal{B}'$ has generators in terms of Lemma \ref{lemma:generator_fan}, which we now want to describe in detail.
		First notice that $x_k' :=x_k$ for $k \in[1,2n]$ are also generators  in $\mathcal{B}'$, because a diagonal flip does not change the boundary. \\
		\textit{Claim:} Every path from $2$ to $2n$ in $\mathcal{B}'$ factors through $z_2'$, with
		\begin{align}\label{eq:new_basis_element_z}
		z_2' := \delta_1 \gamma' \delta_2 .
		\end{align}
		\textit{Proof of the Claim.} We can assume w.l.o.g. that $z_2'$ does not contain a cycle (otherwise, by removing the cycle, every path from $2$ to $2n$ would still factor through it). Furthermore, because the flip does not change the rest of the arrows (apart from the internal arrows $\alpha_{j-3}$, $\alpha_{j-2}$, $\beta_{j-3}$, $\beta_{j-2}$, $\gamma_{j-3}$ and $\gamma_{j-2}$), every path from $2$ to $2n$ still factors through $\delta_1$ and $\delta_2$ in $\mathcal{B}'$. \\
		The generator $z_2'$ for paths from $2$ to $2n$ must contain at least one arrow of the new quadrilateral, because otherwise this generator would already have existed in the original dimer algebra, a contradiction to the fact that all paths from $2$ to $2n$ factor through  $z_2=\delta_1 \alpha_1 \beta_1 \delta_2$ in $\mathcal{B}(Q_F(n))$.\\
		The arrows $\alpha'$ or $\beta'$ immediately lead to a cycle at $i_{j-3}$ or $i_{j-1}$ respectively, so they can't be part of the generator $z_2'$. As $\beta'$ is not part of $z_2'$, if $\gamma'$ is part of the generator, it has to be the only arrow of the new quadrilateral.
		
		Assume now that $\gamma'$ is not part of $z_2'$. Then at least one of the arrows $\alpha''$,  $\beta''$ or $\gamma''$ has to be part of the generator $z_2'$.\\
		If $\alpha''$ was part of it, then either $\alpha'$ or $\beta''$ were also part of the generator. The former is a contradiction as mentioned above, the latter to $z_2'$ using an arrow of the new quadrilateral, as  
		\begin{align*}
		\alpha'' \beta'' \overset{\gamma''}{\cong} x_{2j-1}'x_{2j}'.
		\end{align*}
		The other cases lead to contradictions similarly. \\
		Analogously, we can define
		\begin{align*}
		z_{2j+2}'&:= \gamma_{j-1} \beta' \beta''\\
		z_{2j}'&:= \gamma''\\
		z_{2j-2}'&:= \alpha'' \alpha' \beta_{j-4}.
		\end{align*}
		Furthermore, as every path, which does not contain arrows of the new quadrilateral remains unchanged, we can define $z_{2k}' :=z_{2k}$ for all $k\in[1,n] \setminus \{1,j-1,j,j+1\}$. Then the set $\left \{ \{x_i'\}_{1\leq i \leq 2n}, \{z_{2j}' \}_{1\leq j \leq n} \right \}$ generates $\mathcal{B}'$ in a minimal way  as in Remark after Lemma \ref{lemma:generator_fan}.\\
		It remains to show, that the relations are fulfilled in $\mathcal{B}'$. 
		Of course, the relations $x_{2k+1}' x_{2k+2}'z_{2k+2}'=z_{2k}'x_{2k-1}'x_{2k}'$ hold in $\mathcal{B}'$ by Remark \ref{prop:equivalent_cycle}. \\
		We have to check the relations
		\begin{align*}
		z_{2i}'z_{2i-2}'  =  x_{2i+1}'x_{2i+2}' \ldots x_{2i+2\cdot(n-2)}'.
		\end{align*}
		for all paths involving arrows of the new quadrilateral.
		Let $i=j$, then
		\begin{align*}
		z_{2j}'z_{2j-2}'  \overset{\beta''}{\cong} x_{2j+1}'x_{2j+2}' \gamma_{j-1}\beta'\alpha' \beta_{j-4}  \overset{\gamma'}{\cong}  x_{2j+1}'x_{2j+2}' \delta_2 x_1' x_2' \delta_1 \beta_{j-4} = x_{2j+1}x_{2j+2} \delta_2 x_1 x_2 \delta_1 \beta_{j-4} .
		\end{align*}
		This last path does not contain an arrow of the new quadrilateral, hence we can apply Proposition \ref{prop:relation} and get the desired result:
		\begin{align*}
		x_{2j+1}x_{2j+2} \delta_2 x_1 x_2 \delta_1 \beta_{j-4} \overset{\ref{prop:relation}}{\cong} x_{2j+1}x_{2j+2} \ldots x_{2j-4} = x_{2j+1}'x_{2j+2}' \ldots x_{2j-4}'.
		\end{align*}
		All further relations involving arrows of the quadrilateral follow analogously. 
	\end{proof}
	A direct consequence of the proof of Lemma \ref{lm:flip_equivalence} is:
	\begin{cor}
		The isomorphism of Lemma \ref{lm:flip_equivalence} is induced by
		\begin{align}
		x_k \mapsto x_k' \mbox{ for $k \in [1,2n]$} \\
		z_{2k} \mapsto z_{2k}' \mbox{ for $k \in [1,n]$}.
		\end{align}
	\end{cor}
	
	\begin{theorem}\label{th:flip}
		Let $Q$ be the quiver of the GL$_2$-dimer of an arbitrary triangulation of the $n$-gon, with dimer algebra $\Lambda_{Q}$ and $e_{b'}$ the sum of the boundary idempotents of $\Lambda_{Q}$.
		Then there is an isomorphism
		\begin{align*}
		e_{b'} \Lambda_{Q} e_{b'} \cong e_b \Lambda_{Q_F(n)} e_b.  
		\end{align*}
	\end{theorem}
	\begin{proof}
		We prove the claim by induction over the number of flips.\\
		The induction basis is Lemma \ref{lm:flip_equivalence}.\\
		For the induction step, we use that we can reach any triangulation of an $n$-gon by a finite number of diagonal flips. So let $Q$ be the quiver of an arbitrary triangulation and
		\begin{align*}
		Q = \mu_t \mu_{t-1} \ldots \mu_1 Q_F(n),
		\end{align*}
		where $\mu_1,\ldots,\mu_t$ are $t$ flips of diagonals (writing flips from right to left).
		
		By the induction hypothesis, we know that
		\begin{align*}
		\mathcal{B} (\mu_{t-1} \ldots \mu_1 Q_F(n)) \cong \mathcal{B}(Q_F(n)).
		\end{align*} 
		The induction step follows in a similar way as in the proof of Lemma \ref{lm:flip_equivalence}. Note that Figure \ref{fig:general_flip_dimer_algebra} illustrates the effect of an arbitrary flip, the paths $\delta_1,\ldots,\delta_8$ from and to the quadrilateral involved may differ from those in Lemma \ref{lm:flip_equivalence} in general. However, the arguments for checking do not change. Hence we get the  desired result,
		\begin{align*}
		\mathcal{B}(\mu_{t} \mu_{t-1} \ldots \mu_1 Q_F(n)) \cong \mathcal{B} (\mu_{t-1} \ldots \mu_1 Q_F(n)).
		\end{align*}
		
		\begin{center}
			\includegraphics[scale=0.85]{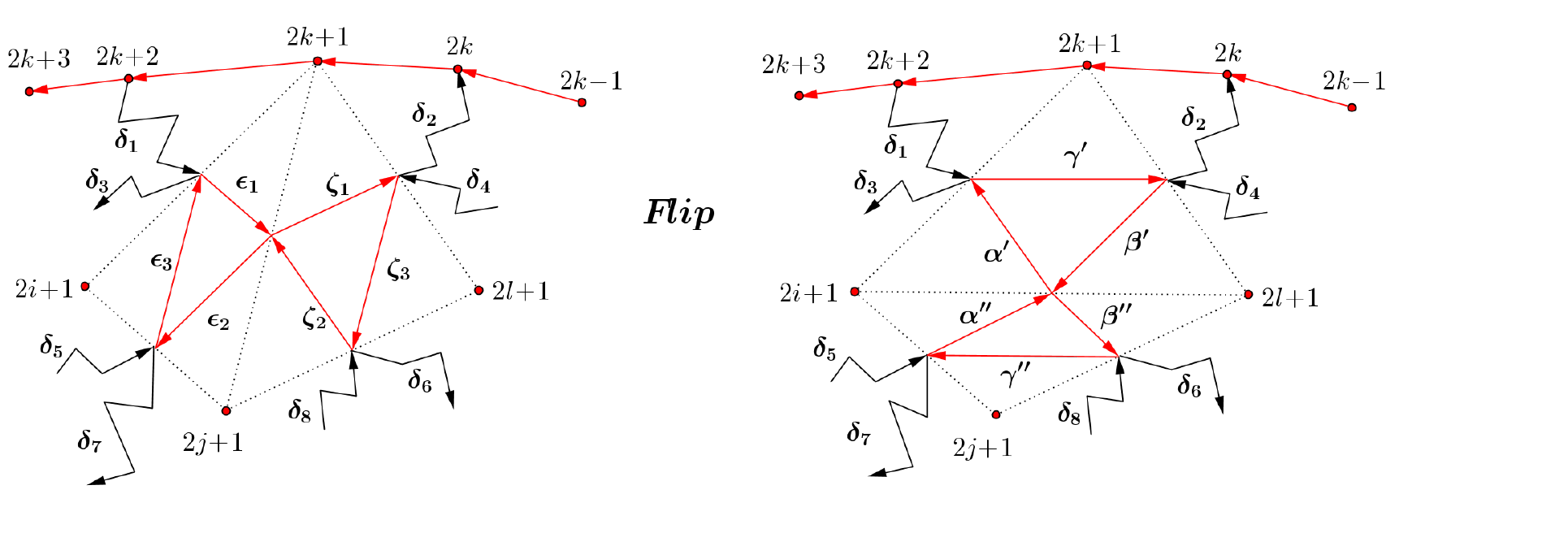}
			\vspace{-1 cm}
			\captionof{figure}{Effect of an arbitrary flip on arrows of type $z_{2k}$.\label{fig:general_flip_dimer_algebra}}
		\end{center} 
	\end{proof}
	\begin{cor}\label{cor:central_element}
		Consider the boundary algebra $\mathcal{B}_Q$ of a dimer model $Q$ of a GL$_2$-dimer of an arbitrary triangulation of the $n$-gon. 
		Then the element $t$,
		\begin{align*}
		t := \sum_{i=1}^{n}{x_{2i-1}x_{2i}z_{2i}}+ \sum_{i=1}^{n}{x_{2i}z_{2i}x_{2i-1}},
		\end{align*}
		is a central element of this algebra.
	\end{cor}
	\begin{proof}
		The element $t$ is the sum of exactly one chordless cycle for every boundary vertex and hence commutes with every element of $\mathcal{B}_Q$.
	\end{proof}


	\section{The general case}
	In this section, we describe the boundary algebras for arbitrary $m$.
	
	Starting with the GL$_m$-dimer as defined in Section \ref{sec:gen_set}, we can reduce the dimer and achieve the quiver of an $n$-gon equivalently as for $m=2$ for arbitrary $m$. From now on, we will assume that the dimer (and hence the quiver) is always reduced.
	
	Figure \ref{fig:quiver_gl_5_quadrilateral} shows the quiver of the GL$_5$-dimer of a triangulation of the quadrilateral. This example already shows, that there is a new type of generators arising for describing the boundary algebra: Let's have an informal look at boundary vertex $10$. In contrast to the case where $m=2$, there are not only paths to $11$ and $9$ but also an additional path to $2$ which cannot be reduced by any of the relations. Hence we need a third type of generators. 
	\begin{figure}[!htb]
		\begin{center}
			\begin{tikzpicture}[scale=1.2]
			\node(1)[label=above: $1$] at (5,5){};
			\node(2)[label=above: $2$] at (4,5){};
			\node(3)[label=above: $3$] at (3, 5){};
			\node(4)[label=above: $4$] at (2, 5){};
			\node(5)[label=above: $5$] at (1, 5){};
			\node(6)[label=above: $6$] at (0, 5){};
			\node(7)[label=left: $7$] at (0, 4){};
			\node(8)[label=left: $8$] at (0, 3){};
			\node(9)[label=left: $9$] at (0, 2){};
			\node(10)[label=left: $10$] at (0, 1){};
			\node(11)[label=left: $11$] at (0, 0){};
			\node(12)[label=below: $12$] at (1,0){};
			\node(13)[label=below: $13$] at (2,0){};
			\node(14)[label=below: $14$] at (3, 0){};
			\node(15)[label=below: $15$] at (4, 0){};
			\node(16)[label=below: $16$] at (5, 0){};
			\node(17)[label=right: $17$] at (5, 1){};
			\node(18)[label=right: $18$] at (5, 2){};
			\node(19)[label=right: $19$] at (5, 3){};
			\node(20)[label=right: $20$] at (5, 4){};
			\node(i1) at (1, 4){};
			\node(i2) at (2, 4){};
			\node(i3) at (3, 4){};
			\node(i4) at (4, 4){};
			\node(i5) at (1, 3){};
			\node(i6) at (2, 3){};
			\node(i7) at (3, 3){};
			\node(i8) at (4, 3){};
			\node(i9) at (1, 2){};
			\node(i10) at (2, 2){};
			\node(i11) at (3, 2){};
			\node(i12) at (4, 2){};
			\node(i13) at (1, 1){};
			\node(i14) at (2, 1){};
			\node(i15) at (3, 1){};
			\node(i16) at (4, 1){};
			
			\path[->] (1) edge (2);
			\path[->] (2) edge (3);
			\path[->] (3) edge (4);
			\path[->] (4) edge (5);
			\path[->] (5) edge (6);
			\path[->] (6) edge (7);
			\path[->] (7) edge (8);
			\path[->] (8) edge (9);
			\path[->] (9) edge (10);
			\path[->] (10) edge (11);
			\path[->] (11) edge (12);
			\path[->] (12) edge (13);
			\path[->] (13) edge (14);
			\path[->] (14) edge (15);
			\path[->] (15) edge (16);
			\path[->] (16) edge (17);
			\path[->] (17) edge (18);
			\path[->] (18) edge (19);
			\path[->] (19) edge (20);
			\path[->] (20) edge (1);
			
			\path[->] (7) edge (5);
			\path[->] (8) edge (i1);
			\path[->] (i1) edge (4);
			\path[->] (9) edge (i5);
			\path[->] (i5) edge (i2);
			\path[->] (i2) edge (3);
			\path[->] (10) edge (i9);
			\path[->] (i9) edge (i6);
			\path[->] (i6) edge (i3);
			\path[->] (i3) edge (2);
			\path[->] (20) edge (i8);
			\path[->] (i8) edge (i11);
			\path[->] (i11) edge (i14);
			\path[->] (i14) edge (12);
			\path[->] (19) edge (i12);
			\path[->] (i12) edge (i15);
			\path[->] (i15) edge (13);
			\path[->] (18) edge (i16);
			\path[->] (i16) edge (14);
			\path[->] (17) edge (15);
			
			\path[->] (2) edge (i4);
			\path[->] (i4) edge (20);
			\path[->] (3) edge (i3);
			\path[->] (i3) edge (i7);
			\path[->] (i7) edge (i8);
			\path[->] (i8) edge (19);
			\path[->] (4) edge (i2);
			\path[->] (i2) edge (i6);
			\path[->] (i6) edge (i10);
			\path[->] (i10) edge (i11);
			\path[->] (i11) edge (i12);
			\path[->] (i12) edge (18);
			\path[->] (5) edge (i1);
			\path[->] (i1) edge (i5);
			\path[->] (i5) edge (i9);
			\path[->] (i9) edge (i13);
			\path[->] (i13) edge (i14);
			\path[->] (i14) edge (i15);
			\path[->] (i15) edge (i16);
			\path[->] (i16) edge (17);
			
			\path[->] (12) edge (i13);
			\path[->] (i13) edge (10);
			\path[->] (13) edge (i14);
			\path[->] (i14) edge (i10);
			\path[->] (i10) edge (i9);
			\path[->] (i9) edge (9);
			\path[->] (14) edge (i15);
			\path[->] (i15) edge (i11);
			\path[->] (i11) edge (i7);
			\path[->] (i7) edge (i6);
			\path[->] (i6) edge (i5);
			\path[->] (i5) edge (8);
			\path[->] (15) edge (i16);
			\path[->] (i16) edge (i12);
			\path[->] (i12) edge (i8);
			\path[->] (i8) edge (i4);
			\path[->] (i4) edge (i3);
			\path[->] (i3) edge (i2);
			\path[->] (i2) edge (i1);
			\path[->] (i1) edge (7);
			
			\draw[fill=black] (1) circle (1pt);
			\draw[fill=black] (2) circle (1pt);
			\draw[fill=black] (3) circle (1pt);
			\draw[fill=black] (4) circle (1pt);
			\draw[fill=black] (5) circle (1pt);
			\draw[fill=black] (6) circle (1pt);
			\draw[fill=black] (7) circle (1pt);
			\draw[fill=black] (8) circle (1pt);
			\draw[fill=black] (9) circle (1pt);
			\draw[fill=black] (10) circle (1pt);
			\draw[fill=black] (11) circle (1pt);
			\draw[fill=black] (12) circle (1pt);
			\draw[fill=black] (13) circle (1pt);
			\draw[fill=black] (14) circle (1pt);
			\draw[fill=black] (15) circle (1pt);
			\draw[fill=black] (16) circle (1pt);
			\draw[fill=black] (17) circle (1pt);
			\draw[fill=black] (18) circle (1pt);
			\draw[fill=black] (19) circle (1pt);
			\draw[fill=black] (20) circle (1pt);
			\draw[fill=black] (i1) circle (1pt);
			\draw[fill=black] (i2) circle (1pt);
			\draw[fill=black] (i3) circle (1pt);
			\draw[fill=black] (i4) circle (1pt);
			\draw[fill=black] (i5) circle (1pt);
			\draw[fill=black] (i6) circle (1pt);
			\draw[fill=black] (i7) circle (1pt);
			\draw[fill=black] (i8) circle (1pt);
			\draw[fill=black] (i9) circle (1pt);
			\draw[fill=black] (i10) circle (1pt);
			\draw[fill=black] (i11) circle (1pt);
			\draw[fill=black] (i12) circle (1pt);
			\draw[fill=black] (i13) circle (1pt);
			\draw[fill=black] (i14) circle (1pt);
			\draw[fill=black] (i15) circle (1pt);
			\draw[fill=black] (i16) circle (1pt);
			\end{tikzpicture}
			\vspace{-0.32 cm}
			\caption{Reduced quiver of the GL$_5$-dimer of a quadrilateral with diagonal incident with $1$. \label{fig:quiver_gl_5_quadrilateral}}
		\end{center}
	\end{figure}
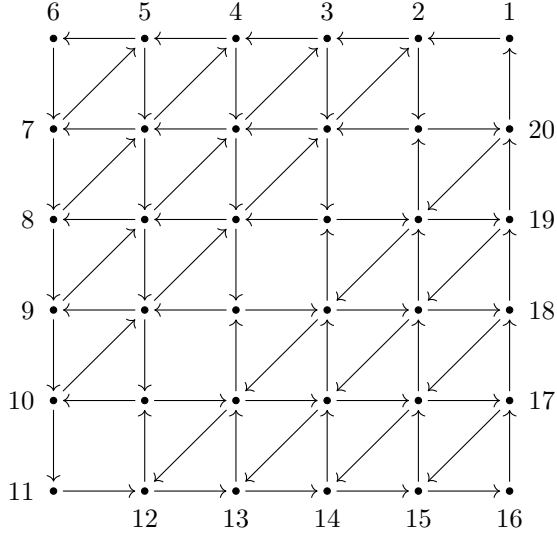
	In order to introduce a formal notation for the generators, we have a closer look at the internal vertices. We can give a formula for the number of internal vertices, depending on $m$ and $n$.
	\begin{defn}[polygonal numbers of second order] For a polygon with $s$ vertices, the $k^{th}$ $s$-gonal number of second order is
		\begin{align*}
		P_2(s,k)= \frac{k^2\cdot (s-2) + k \cdot (s-4)}{2}.
		\end{align*}
	\end{defn}
	\begin{rem}
		Using the setting as in definition above, the polygonal number (of first order) is defined as $P(s,k)= \frac{k^2\cdot (s-2) - k \cdot (s-4)}{2}$, see A057145 in the OEIS \cite{oeis}.
	\end{rem}
	
	\begin{prop}
		The number of internal vertices of the dimer model of the GL$_m$-dimer of the fan triangulation of the $n$-gon is $P_2(n,m-1)$ for $n>3$.
	\end{prop}
	\begin{proof}
		The proof is done by induction on $n$ and $m$. First let $m$ be fixed. For induction basis let $n=4$. The number of internal vertices equals $(m-1)^2$ by construction, which coincides with $P_2(4,m-1)$. Let the number of internal vertices $V_{n,m}$ of the GL$_m$-dimer of the fan triangulation of the $n$-gon be $P_2(n,m-1)$. Consider the $n+1$-gon, where we add a triangle to the fan. By construction of the GL$_m$-dimer, we get
		\begin{align*}
		1+2+\ldots +(m-2)+(m-1)
		\end{align*} 
		additional vertices by adding a triangle.
		By using the induction hypothesis,
		\begin{align*}
		V_{n+1,m}&= P_2(n,m-1) + 1+2+ \ldots + (m-2)+m-1=\\
		&= \frac{(m-1)^2 \cdot (n-2) + (m-1) \cdot (n-4)}{2} + \frac{(m-1)\cdot m}{2} =\\
		&= \frac{(m-1)^2 \cdot (n-2) + (m-1) \cdot (n-4)}{2} + \frac{(m-1)^2}{2} + \frac{m-1}{2} = \\
		&= \frac{(m-1)^2 (n-1) + (m-1)\cdot (n-3)}{2} = P_2(n+1,m-1),
		\end{align*}
		and we achieve the desired result. 
		
		Now let $n$ be fixed. For $m=2$, the number of inner vertices of the GL$_2$-dimer of an $n$-gon  is $(n-3)$ and coincides with $P_2(n,2)$. Again, let the number $V_{n,m}$ of internal vertices of the GL$_m$ dimer of the $n$-gon be $P_2(n,m-1)$. If we increase $m$ by one, the number of internal vertices increases by $m \cdot n - 2m -1$, 
		by construction of the GL$_m$-dimer. So the number $V_{n,m+1}$ of internal vertices of quiver of the GL$_{m+1}$-dimer is by using induction hypothesis
		\begin{align*}
		V_{n,m+1}=P_2(n,m-1)+ m\cdot n - 2m-1= \frac{m^2\cdot(n-2)+ m\cdot (n-4)}{2}= P_2(n,m)
		\end{align*}
		and the proof is done.
	\end{proof}
	\begin{rem}
		The number of internal vertices is the same for any triangulation of the $n$-gon.
	\end{rem}
	Before we can state the structure of the quiver $Q_F(m,n)$ of the GL$_m$-dimer of the fan triangulation of the $n$-gon, we have to introduce some notation.\\ 
	The boundary vertices are labelled by $1,\ldots,nm$ anticlockwise.\\
	The quiver $Q_F(m,n)$ contains $m-1$ disjoint nested oriented paths from $2+i$ to $nm-i$ for $i=0,\ldots,m-2$ formed by successive arrows. We will denote them by $\alpha_P$, where $P$ is the sink of the arrow. The internal vertices are labelled by a $3$-tuple $(a,b,c)$ depending on their position along these oriented paths as follows:
	\begin{itemize}
		\item[$a$]$\in [1,n-2]$  denotes the triangle, to which the internal vertex can be assigned to. 
		\item[$b$]$ \in [1,m-1]$ is one less than the starting vertex of the nested path.
		\item[$c$]$\in [1,b]$ counts the number of internal vertices up to $b$ in every triangle.    
	\end{itemize} 
	Figure \ref{fig:dimer_hexagon_gl4} shows the labelling of all internal vertices of $Q_F(4,6)$.
	The arrows are all indexed by their sinks. Along the boundary they are $x_k: k-1 \rightarrow k$. All the other arrows are as follows:
	\begin{notation}\label{def:arrows_general}
		\begin{align*}
		\alpha_{(1,i,1)}&:= i+1 \rightarrow (1,i,1) \ \ i\in [1,m-1] \\
		\alpha_{(a,b,1)}&:= (a-1,b,b) \rightarrow (a,b,1) \ \ a\in [2,n-2], b\in [1,m-2] \\
		\alpha_{(a,b,c)}&:= (a,b.c-1) \rightarrow (a,b,c) \ \ a\in [1,n-2], b\in [2,m-2], c \in [2,b] \\
		\alpha_{m \cdot n - i}&:= (n-2,i+1,i) \rightarrow m \cdot n -i \ \ i \in[1,m-2] \\
		\alpha_{m\cdot n} &:= (n-3,1,1) \rightarrow m\cdot n \\
		\beta_{m} &:= m+2 \rightarrow m \\
		\beta_{i} &:= (1,i,1) \rightarrow i \ \ i \in [2,m-1] \\
		\beta_{(k-1,m-1,i-2)}&:= k\cdot m + i \rightarrow (k-1,m-1,i-2) \ \ i\in [3,m] , k \in [2,n-1] \\
		\beta_{(k-2,m-1,m-1)}&:=k\cdot m + 1 \rightarrow (k-2,m-1,m-1) \\
		\beta_{(a,b,c)} &:= \left\{ \begin{array}{ll}
		(a,b+1,c+1) \rightarrow (a,b,c) & \mbox{$a\in [1,n-2] \  b \in [2,m-1] \ c\in [1,m-3]$} \\
		(a+1,b+1,1) \rightarrow (a,b,c) & \mbox{$a\in [1,n-2] \  b \in [2,m-1] \ c=b$} \end{array} \right. \\
		\gamma_{m \cdot (n-1)} &:= m \cdot (n-1)+2 \rightarrow m \cdot (n-1) \\
		\gamma_{(n-2,i+1,i)}   &:=m \cdot n - i + 1 \rightarrow (n-2,i+1,i) \ \ i\in [1,m-2] \\  
		\gamma_{m \cdot k +1 + i} &:= (k,m-1,i) \rightarrow m \cdot k +1 + i \ \ k \in[1,n-3] i \in [1,m-1] \\
		\gamma_{m \cdot (n-2) + 1 + i } &:= (n-2,m-1,i) \rightarrow m \cdot k +1 + i \ \  i \in [1,m-2] \\
		\gamma_{(a,b,c)} &:= (a,b-1,c) \rightarrow (a,b,c) \ \ a\in [1,n-2], \ b\in[2,m-1], \ c = b-1 
		\end{align*}
	\end{notation}
	In Figure \ref{fig:dimer_hexagon_gl4} the definition of the different types of arrows is shown in the case of the GL$_4$-dimer of a hexagon: Beside the black boundary arrows $x_k$ $k\in [1,m\cdot n]$, the remaining arrows can be identified as follows: The bold black arrows are named $\alpha_i$, the blue arrows $\beta_i$ and the green arrows $\gamma_i$, where the index $i$ always coincides with the labelling of the sink of the arrow. This means $i$ is either a triple $(a,b,c)$ or some natural number in $[1,m\cdot n]$ with $i \not \equiv 1 \mod m$. Note that additionally to the arrows and vertices of the quiver the diagonals of the original triangulation are drawn as dotted lines to emphasize the idea of labelling the internal vertices $(a,b,c)$. 
	\vspace{-1.3 cm}
	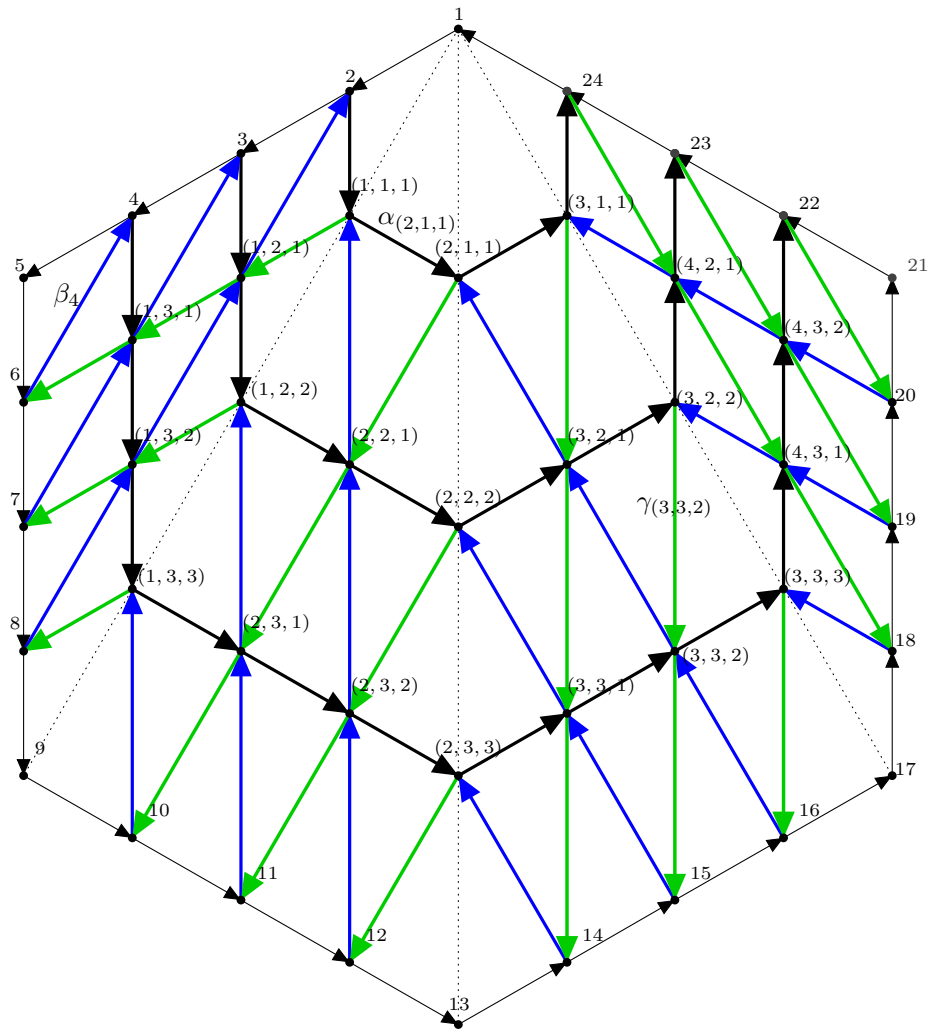
\begin{figure}[!htb]
		\begin{center}
			\definecolor{qqccqq}{rgb}{0,0.8,0}
			\definecolor{qqqqff}{rgb}{0,0,1}
			\definecolor{ffqqqq}{rgb}{1,0,0}
			\definecolor{uququq}{rgb}{0.25,0.25,0.25}
			\begin{tikzpicture}[line cap=round,line join=round,>=triangle 45,x=1.0cm,y=1.0cm, scale=1.1]
			\clip(-2,-6) rectangle (14.35,8.26);
			\draw [line width=0.4pt,dotted] (5.2,7)-- (0,-2);
			\draw [line width=0.4pt,dotted] (5.2,7)-- (5.2,-5);
			\draw [line width=0.4pt,dotted] (5.2,7)-- (10.39,-2);
			\draw [->] (5.2,7) -- (3.9,6.25);
			\draw [->] (3.9,6.25) -- (2.6,5.5);
			\draw [->] (2.6,5.5) -- (1.3,4.75);
			\draw [->] (1.3,4.75) -- (0,4);
			\draw [->] (0,4) -- (0,2.5);
			\draw [->] (0,2.5) -- (0,1);
			\draw [->] (0,1) -- (0,-0.5);
			\draw [->] (0,-0.5) -- (0,-2);
			\draw [->] (0,-2) -- (1.3,-2.75);
			\draw [->] (1.3,-2.75) -- (2.6,-3.5);
			\draw [->] (2.6,-3.5) -- (3.9,-4.25);
			\draw [->] (3.9,-4.25) -- (5.2,-5);
			\draw [->] (5.2,-5) -- (6.5,-4.25);
			\draw [->] (6.5,-4.25) -- (7.79,-3.5);
			\draw [->] (7.79,-3.5) -- (9.09,-2.75);
			\draw [->] (9.09,-2.75) -- (10.39,-2);
			\draw [->] (10.39,-2) -- (10.39,-0.5);
			\draw [->] (10.39,-0.5) -- (10.39,1);
			\draw [->] (10.39,1) -- (10.39,2.5);
			\draw [->] (10.39,2.5) -- (10.39,4);
			\draw [->] (10.39,4) -- (9.09,4.75);
			\draw [->] (9.09,4.75) -- (7.79,5.5);
			\draw [->] (7.79,5.5) -- (6.5,6.25);
			\draw [->] (6.5,6.25) -- (5.2,7);
			\draw [->,line width=1.2pt] (3.9,6.25) -- (3.9,4.75);
			\draw [->,line width=1.2pt] (3.9,4.75) -- node[label=above:$\alpha_{(2,1,1)}$,pos=0.625]{} (5.2,4);
			\draw [->,line width=1.2pt] (5.2,4) -- (6.5,4.75);
			\draw [->,line width=1.2pt] (6.5,4.75) -- (6.5,6.25);
			\draw [->,line width=1.2pt] (2.6,5.5) -- (2.6,4);
			\draw [->,line width=1.2pt] (2.6,4) -- (2.6,2.5);
			\draw [->,line width=1.2pt] (2.6,2.5) -- (3.9,1.75);
			\draw [->,line width=1.2pt] (3.9,1.75) -- (5.2,1);
			\draw [->,line width=1.2pt] (5.2,1) -- (6.5,1.75);
			\draw [->,line width=1.2pt] (6.5,1.75) -- (7.79,2.5);
			\draw [->,line width=1.2pt] (7.79,2.5) -- (7.79,4);
			\draw [->,line width=1.2pt] (7.79,4) -- (7.79,5.5);
			\draw [->,line width=1.2pt] (1.3,4.75) -- (1.3,3.25);
			\draw [->,line width=1.2pt] (1.3,3.25) -- (1.3,1.75);
			\draw [->,line width=1.2pt] (1.3,1.75) -- (1.3,0.25);
			\draw [->,line width=1.2pt] (1.3,0.25) -- (2.6,-0.5);
			\draw [->,line width=1.2pt] (2.6,-0.5) -- (3.9,-1.25);
			\draw [->,line width=1.2pt] (3.9,-1.25) -- (5.2,-2);
			\draw [->,line width=1.2pt] (5.2,-2) -- (6.5,-1.25);
			\draw [->,line width=1.2pt] (6.5,-1.25) -- (7.79,-0.5);
			\draw [->,line width=1.2pt] (7.79,-0.5) -- (9.09,0.25);
			\draw [->,line width=1.2pt] (9.09,0.25) -- (9.09,1.75);
			\draw [->,line width=1.2pt] (9.09,1.75) -- (9.09,3.25);
			\draw [->,line width=1.2pt] (9.09,3.25) -- (9.09,4.75);
			\draw [->,line width=1.2pt,color=qqqqff] (0,2.5) -- node[label=above:\textcolor{black}{\textbf{$\beta_4$}},pos=0.4]{} (1.3,4.75);
			\draw [->,line width=1.2pt,color=qqqqff] (0,1) -- (1.3,3.25);
			\draw [->,line width=1.2pt,color=qqqqff] (1.3,3.25) -- (2.6,5.5);
			\draw [->,line width=1.2pt,color=qqqqff] (0,-0.5) -- (1.3,1.75);
			\draw [->,line width=1.2pt,color=qqqqff] (1.3,1.75) -- (2.6,4);
			\draw [->,line width=1.2pt,color=qqqqff] (2.6,4) -- (3.9,6.25);
			\draw [->,line width=1.2pt,color=qqqqff] (3.9,-4.25) -- (3.9,-1.25);
			\draw [->,line width=1.2pt,color=qqqqff] (3.9,-1.25) -- (3.9,1.75);
			\draw [->,line width=1.2pt,color=qqqqff] (3.9,1.75) -- (3.9,4.75);
			\draw [->,line width=1.2pt,color=qqqqff] (2.6,-3.5) -- (2.6,-0.5);
			\draw [->,line width=1.2pt,color=qqqqff] (2.6,-0.5) -- (2.6,2.5);
			\draw [->,line width=1.2pt,color=qqqqff] (1.3,-2.75) -- (1.3,0.25);
			\draw [->,line width=1.2pt,color=qqqqff] (6.5,-4.25) -- (5.2,-2);
			\draw [->,line width=1.2pt,color=qqqqff] (7.79,-3.5) -- (6.5,-1.25);
			\draw [->,line width=1.2pt,color=qqqqff] (6.5,-1.25) -- (5.2,1);
			\draw [->,line width=1.2pt,color=qqqqff] (9.09,-2.75) -- (7.79,-0.5);
			\draw [->,line width=1.2pt,color=qqqqff] (7.79,-0.5) -- (6.5,1.75);
			\draw [->,line width=1.2pt,color=qqqqff] (6.5,1.75) -- (5.2,4);
			\draw [->,line width=1.2pt,color=qqqqff] (10.39,-0.5) -- (9.09,0.25);
			\draw [->,line width=1.2pt,color=qqqqff] (10.39,1) -- (9.09,1.75);
			\draw [->,line width=1.2pt,color=qqqqff] (9.09,1.75) -- (7.79,2.5);
			\draw [->,line width=1.2pt,color=qqqqff] (10.39,2.5) -- (9.09,3.25);
			\draw [->,line width=1.2pt,color=qqqqff] (9.09,3.25) -- (7.79,4);
			\draw [->,line width=1.2pt,color=qqqqff] (7.79,4) -- (6.5,4.75);
			\draw [->,line width=1.3pt,color=qqccqq] (9.09,4.75) -- (10.39,2.5);
			\draw [->,line width=1.3pt,color=qqccqq] (7.79,5.5) -- (9.09,3.25);
			\draw [->,line width=1.3pt,color=qqccqq] (9.09,3.25) -- (10.39,1);
			\draw [->,line width=1.3pt,color=qqccqq] (6.5,6.25) -- (7.79,4);
			\draw [->,line width=1.3pt,color=qqccqq] (7.79,4) -- (9.09,1.75);
			\draw [->,line width=1.3pt,color=qqccqq] (9.09,1.75) -- (10.39,-0.5);
			\draw [->,line width=1.3pt,color=qqccqq] (6.5,4.75) -- (6.5,1.75);
			\draw [->,line width=1.3pt,color=qqccqq] (6.5,1.75) -- (6.5,-1.25);
			\draw [->,line width=1.3pt,color=qqccqq] (6.5,-1.25) -- (6.5,-4.25);
			\draw [->,line width=1.3pt,color=qqccqq] (7.79,2.5) -- node[label=above:\textcolor{black}{\textbf{${\gamma_{(3,3,2)}}$}},pos=0.55]{} (7.79,-0.5);
			\draw [->,line width=1.3pt,color=qqccqq] (7.79,-0.5) -- (7.79,-3.5);
			\draw [->,line width=1.3pt,color=qqccqq] (9.09,0.25) -- (9.09,-2.75);
			\draw [->,line width=1.3pt,color=qqccqq] (5.2,4) -- (3.9,1.75);
			\draw [->,line width=1.3pt,color=qqccqq] (3.9,1.75) -- (2.6,-0.5);
			\draw [->,line width=1.3pt,color=qqccqq] (2.6,-0.5) -- (1.3,-2.75);
			\draw [->,line width=1.3pt,color=qqccqq] (5.2,1) -- (3.9,-1.25);
			\draw [->,line width=1.3pt,color=qqccqq] (3.9,-1.25) -- (2.6,-3.5);
			\draw [->,line width=1.3pt,color=qqccqq] (5.2,-2) -- (3.9,-4.25);
			\draw [->,line width=1.3pt,color=qqccqq] (3.9,4.75) -- (2.6,4);
			\draw [->,line width=1.3pt,color=qqccqq] (2.6,4) -- (1.3,3.25);
			\draw [->,line width=1.3pt,color=qqccqq] (1.3,3.25) -- (0,2.5);
			\draw [->,line width=1.3pt,color=qqccqq] (2.6,2.5) -- (1.3,1.75);
			\draw [->,line width=1.3pt,color=qqccqq] (1.3,1.75) -- (0,1);
			\draw [->,line width=1.3pt,color=qqccqq] (1.3,0.25) -- (0,-0.5);
			\begin{scriptsize}
			\fill [color=black] (0,4) circle (1.5pt);
			\draw[color=black] (0.2,4.15) node[label=left: $5$]{};
			\fill [color=black] (0,-2) circle (1.5pt);
			\draw[color=black] (0.2,-1.67) node {$9$};
			\fill [color=black] (5.2,-5) circle (1.5pt);
			\draw[color=black] (5.21,-4.75) node {$13$};
			\fill [color=black] (10.39,-2) circle (1.5pt);
			\draw[color=black] (10.55,-1.92) node {$17$};
			\fill [color=uququq] (10.39,4) circle (1.5pt);
			\draw[color=uququq] (10.7,4.15) node {$21$};
			\fill [color=black] (5.2,7) circle (1.5pt);
			\draw[color=black] (5.21,7.18) node {$1$};
			\fill [color=black] (2.6,5.5) circle (1.5pt);
			\draw[color=black] (2.62,5.67) node {$3$};
			\fill [color=black] (0,1) circle (1.5pt);
			\draw[color=black] (-0.1,1.34) node {$7$};
			\fill [color=black] (2.6,-3.5) circle (1.5pt);
			\draw[color=black] (2.92,-3.17) node {$11$};
			\fill [color=black] (7.79,-3.5) circle (1.5pt);
			\draw[color=black] (8.1,-3.17) node {$15$};
			\fill [color=black] (10.39,1) circle (1.5pt);
			\draw[color=black] (10.55,1.09) node {$19$};
			\fill [color=black] (7.79,5.5) circle (1.5pt);
			\draw[color=black] (8.1,5.62) node {$23$};
			\fill [color=black] (1.3,4.75) circle (1.5pt);
			\draw[color=black] (1.32,4.95) node {$4$};
			\fill [color=black] (3.9,6.25) circle (1.5pt);
			\draw[color=black] (3.91,6.43) node {$2$};
			\fill [color=black] (6.5,6.25) circle (1.5pt);
			\draw[color=black] (6.81,6.38) node {$24$};
			\fill [color=black] (9.09,4.75) circle (1.5pt);
			\draw[color=black] (9.4,4.9) node {$22$};
			\fill [color=black] (10.39,2.5) circle (1.5pt);
			\draw[color=black] (10.55,2.59) node {$20$};
			\fill [color=black] (10.39,-0.5) circle (1.5pt);
			\draw[color=black] (10.55,-0.36) node {$18$};
			\fill [color=black] (9.09,-2.75) circle (1.5pt);
			\draw[color=black] (9.4,-2.42) node {$16$};
			\fill [color=black] (6.5,-4.25) circle (1.5pt);
			\draw[color=black] (6.81,-3.92) node {$14$};
			\fill [color=black] (3.9,-4.25) circle (1.5pt);
			\draw[color=black] (4.22,-3.92) node {$12$};
			\fill [color=black] (1.3,-2.75) circle (1.5pt);
			\draw[color=black] (1.62,-2.42) node {$10$};
			\fill [color=black] (0,-0.5) circle (1.5pt);
			\draw[color=black] (-0.1,-0.16) node {$8$};
			\fill [color=black] (0,2.5) circle (1.5pt);
			\draw[color=black] (-0.1,2.84) node {$6$};
			\fill [color=black] (3.9,4.75) circle (1.5pt);
			\draw[color=black] (4.32,5.1) node {$(1,1,1)$};
			\fill [color=black] (5.2,4) circle (1.5pt);
			\draw[color=black] (5.32,4.35) node {$(2,1,1)$};
			\fill [color=black] (6.5,4.75) circle (1.5pt);
			\draw[color=black] (6.91,4.9) node {$(3,1,1)$};
			\fill [color=uququq] (6.5,6.25) circle (1.5pt);
			\fill [color=black] (2.6,2.5) circle (1.5pt);
			\draw[color=black] (3.12,2.64) node {$(1,2,2)$};
			\fill [color=black] (5.2,1) circle (1.5pt);
			\draw[color=black] (5.32,1.34) node {$(2,2,2)$};
			\fill [color=black] (7.79,2.5) circle (1.5pt);
			\draw[color=black] (8.21,2.54) node {$(3,2,2)$};
			\fill [color=uququq] (7.79,5.5) circle (1.5pt);
			\fill [color=black] (1.3,0.25) circle (1.5pt);
			\draw[color=black] (1.77,0.39) node {$(1,3,3)$};
			\fill [color=black] (5.2,-2) circle (1.5pt);
			\draw[color=black] (5.32,-1.67) node {$(2,3,3)$};
			\fill [color=black] (9.09,0.25) circle (1.5pt);
			\draw[color=black] (9.5,0.39) node {$(3,3,3)$};
			\fill [color=uququq] (9.09,4.75) circle (1.5pt);
			\fill [color=black] (2.6,4) circle (1.5pt);
			\draw[color=black] (3.02,4.35) node {$(1,2,1)$};
			\fill [color=black] (3.9,1.75) circle (1.5pt);
			\draw[color=black] (4.32,2.09) node {$(2,2,1)$};
			\fill [color=black] (6.5,1.75) circle (1.5pt);
			\draw[color=black] (6.91,2.09) node {$(3,2,1)$};
			\fill [color=black] (7.79,4) circle (1.5pt);
			\draw[color=black] (8.21,4.15) node {$(4,2,1)$};
			\fill [color=black] (1.3,3.25) circle (1.5pt);
			\draw[color=black] (1.73,3.59) node {$(1,3,1)$};
			\fill [color=black] (1.3,1.75) circle (1.5pt);
			\draw[color=black] (1.73,2.09) node {$(1,3,2)$};
			\fill [color=black] (2.6,-0.5) circle (1.5pt);
			\draw[color=black] (3.02,-0.16) node {$(2,3,1)$};
			\fill [color=black] (3.9,-1.25) circle (1.5pt);
			\draw[color=black] (4.32,-0.92) node {$(2,3,2)$};
			\fill [color=black] (6.5,-1.25) circle (1.5pt);
			\draw[color=black] (6.91,-0.92) node {$(3,3,1)$};
			\fill [color=black] (7.79,-0.5) circle (1.5pt);
			\draw[color=black] (8.28,-0.56) node {$(3,3,2)$};
			\fill [color=black] (9.09,1.75) circle (1.5pt);
			\draw[color=black] (9.5,1.89) node {$(4,3,1)$};
			\fill [color=black] (9.09,3.25) circle (1.5pt);
			\draw[color=black] (9.5,3.39) node {$(4,3,2)$};
			\end{scriptsize}
			\end{tikzpicture}
			\vspace{-1.7 cm}
			\caption{$Q_F(4,6)$: Quiver of the GL$_4$-dimer of the fan triangulation of a hexagon. For illustration, some arrows are labelled. \label{fig:dimer_hexagon_gl4}}
		\end{center}
	\end{figure} 
	\begin{prop}\label{prop:quiver_fan_triangulation_general}
		Let $Q_F(m,n)$ be the quiver of the fan triangulation of the $n$-gon, $n\geq 3$. Then $Q_F(m,n)$ has the following form: \\
		It consists of $m \cdot n$ vertices on the boundary, labelled anticlockwise by $1 , \ldots, m\cdot n$, and $P_2(n,m-1)$ internal vertices labelled $(a,b,c)$, with $a$,$b$ and $c$ as described above. \\
		Furthermore it has $m\cdot n +2$ arrows between the boundary vertices
		\begin{align*}
		x_k &: k-1 \rightarrow k  \\
		y_m &:= \beta_m \\
		y_{m\cdot (n-1)} &:=\gamma_{m\cdot (n-1)} 
		\end{align*}
		and the internal arrows as described in \ref{def:arrows_general}. 
	\end{prop} 
	\begin{proof}
		The proof is similar to the proof of \ref{prop:quiver_fan_triangulation}. 
	\end{proof}

	Equivalently to former section, we define some elements of the boundary algebra and show, that these elements are generating the algebra.
	
	\begin{defn}\label{def:generators_general}
		We define paths $z_{m\cdot j+k}$ and $y_{m\cdot j + k}$ as follows:
		\allowdisplaybreaks
		\begin{align}
		y_{k} &:= \left(\prod_{i=k}^{m-1}{\beta_{(1,m-1,m-k)}}\right) \cdot \beta_k ,  \ \  k \in [2,m] \\
		y_{m\cdot j + k} &:=         \left(\prod_{i=k}^{m-1}{\beta_{(j+1,m+1-k,m-k)}}\right) \cdot \beta_{(j,k-1,k-1)} \cdot 
		\left(\prod_{i=k}^{m-1}{\gamma_{(j,i,k-1)}}\right) \cdot \gamma_{(m\cdot j + k)}, \ \ j\in [1,n-3] \ \ k \in [2,m] \\
		y_{m\cdot (n-2) + k} &:= \left(\prod_{i=k}^{m-1}{\gamma_{(n-2,i,k-1)}}\right) \cdot \gamma_{(m\cdot(n-2)+k)}, \ \ k \in [2,m] \\
		y_{m\cdot n -k} &:= \left(\prod_{i=1}^{n-2}{ \prod_{j=0}^{k}{\alpha_{((i,k+1,j))}}}\right) \cdot \alpha_{m\cdot n - k}, \ \ k\in [0,m-2] \\
		z_{k} &:= \gamma_{(1,k,1)}  \beta_k, \ \ k\in [2,m-1] \\
		z_{m\cdot j + k} &:= \beta_{(j,m-1,k-1)} \gamma_{m\cdot j + k}, \ \ k\in [2,m-1] \\
		z_{m\cdot n - k} &:= \gamma_{n-2,k+1,k}  \beta_{m\cdot n -k}, \ \ k\in [2,m-1].
		\end{align}
	\end{defn}
	
	\begin{rem}
		The products above might be empty, which happens in equations $(9)-(11)$, when $k=m$. In these cases, only a single arrow remains. These are exactly the arrows
		\begin{align*}
		y_{m}&:= \beta_{m} \\
		y_{m\cdot j + m} &:= \beta_{(j,k-1,k-1)} \cdot \gamma_{m\cdot j + m} \\
		y_{m\cdot (n-2) + m} &:= \gamma_{m\cdot(n-2)+m}. 
		\end{align*}
	\end{rem}	
	
	Note that the notation of the arrows $z_i$ is different to the previous section, because of the new type of generators. \\
	
	We will now define a quiver and show, that the boundary algebra of $Q_F(m,n)$ is isomorphic to it up to some relations.
	\begin{defn}\label{def:quiver_general_case}
		The quiver $\Gamma(m,n)$ is defined by $m\cdot n$ vertices labelled anticlockwise and $3 \cdot n \cdot (m-1)$ arrows
		\begin{align*}
		&x_k:  k-1 \rightarrow k, \ \ \ \ \ \ \ \ k \in [1,m\cdot n]\\
		&y_k:  k+2-2i \rightarrow k,  \ \  k \equiv -i \mod m, \ k\not \equiv 1 \mod m, \ \text{ and } -i \in [0,m-2] \\
		&z_k:  k+1 \rightarrow k, \ \ \ \ \ \ \ \ k \not \equiv 1,0 \mod m.
		\end{align*} 
	\end{defn} 
	Figure \ref{fig:quiver_general} shows the quiver $\Gamma(4,6)$.
	\begin{figure}[!htb]
		\begin{center}
			\includegraphics[scale=0.5]{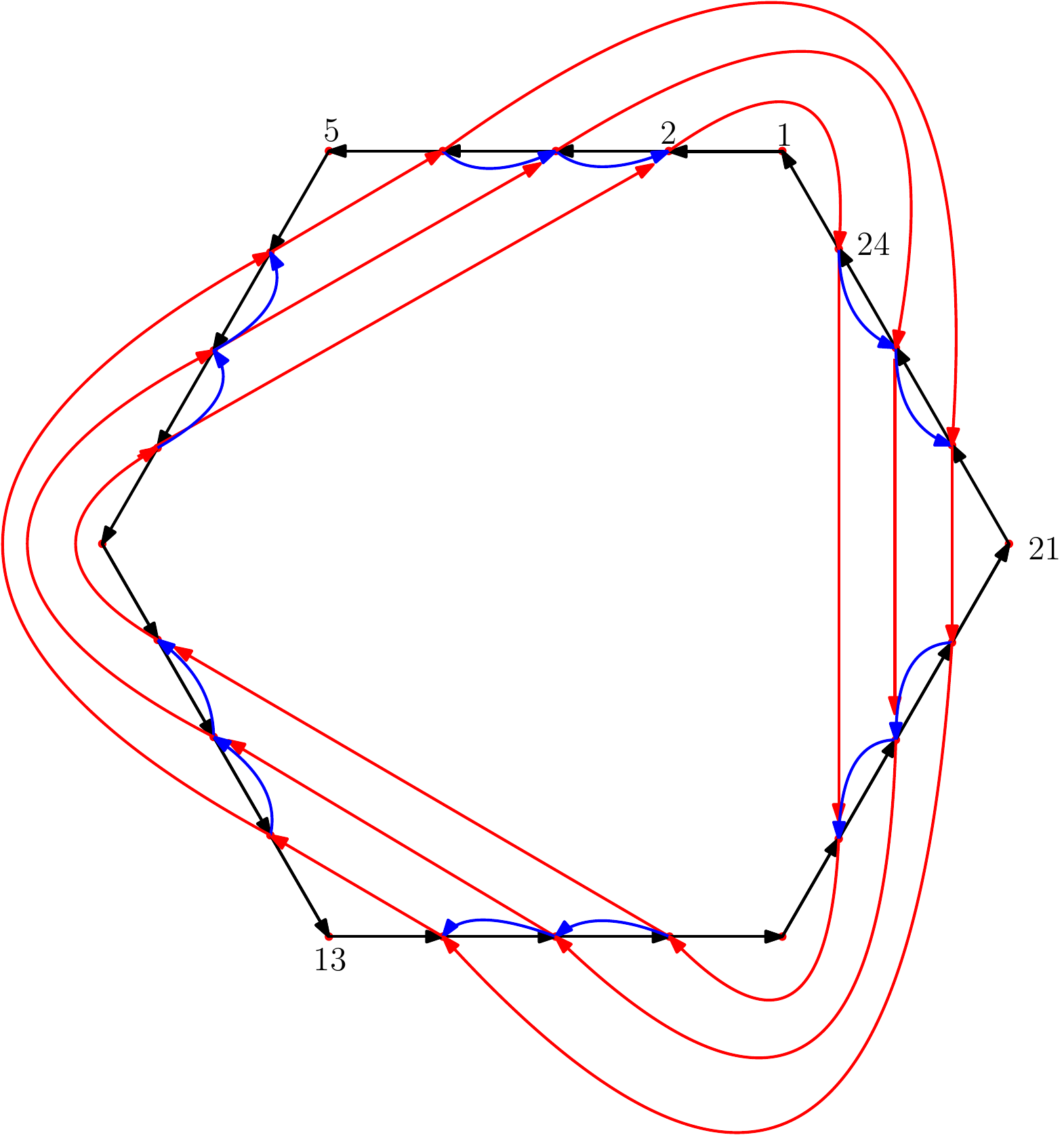}
			\vspace{-0.3 cm}
			\caption{$\Gamma(4,6)$ as in Definition \ref{def:quiver_general_case}. Arrows $x$ are black, arrows $y$ are red, arrows $z$ are blue.\label{fig:quiver_general}}
		\end{center}
	\end{figure} 
	\begin{rem}
		The indices of the tail $l$ and head $k=j\cdot m - k'$ with $k'\in [0,m-2]$ of $y_k$ fulfil
		\begin{align*}
		l+k &= 2 \cdot(j\cdot m +1) \\
		\Leftrightarrow l &= m\cdot j + 2 + k' = k + 2 +2k',
		\end{align*} 
		which is equivalent to the description of $y_k$ in Definition \ref{def:quiver_general_case}.
	\end{rem}
	
	As in the previous section we first describe the boundary algebra of the dimer algebra $\Lambda_{Q_F(m,n)}$, where $Q_F(m,n)$ is the quiver of the GL$_m$-dimer of the fan triangulation of the $n$-gon.
	\begin{prop}[Boundary algebra of the fan triangulation]\label{prop:quiver_fan_general}
		The boundary algebra $\mathcal{B}_{Q_F(m,n)}$ is isomorphic to $\Gamma(m,n)$
		satisfying the following relations obtained by the natural potential $W$:
		\begin{eqnarray*}
			x_{k+2-2i}y_{k} & = &  y_{k+1}z_{k}, \ k \not \equiv 0,1 \mod m, \ k \equiv -i \mod m, \text{ and } -i \in [1,m-2] \\
			x_{k+1}z_{k} &= & z_{k-1}x_{k}, \ k \not \equiv 0,1,2 \mod m \\
			x_{k+1} z_{k} &= & y_{k-2} x_{k-1}x_{k}, \ k \equiv 2 \mod m \\
			x_{k+1} x_{k+2} y_{k} &= & z_{k-1}x_{k}, \ k \equiv 0 \mod m \\
			y_{k+2-2j}y_{k} & = & x_{k+2m+1}x_{k+2m+2} \cdots x_{k} , k \equiv -j \mod m, \ k\not \equiv 1 \mod m, \ \text{ and } -j \in [0,m-2]
		\end{eqnarray*}
		where $k\in [1,m\cdot n]$  and indices always considered modulo $m\cdot n$. 
	\end{prop}
	\begin{proof}
		The proof is done by combining Remark \ref{prop:equivalent_cycle}, Lemma \ref{lemma:generator_fan} and Proposition \ref{prop:relation} applied to the general case.
	\end{proof}
	
	The flip-invariance in case of $m>2$ is a direct consequence of Theorem \ref{th:flip}.
	\begin{theorem}[Flip equivalence]\label{th:flip_equiv_general}
		Let $Q_F(m,n)$ be the quiver of the GL$_m$-dimer of the fan triangulation of the $n$-gon, let $Q'$ be the quiver of the GL$_m$-dimer of an arbitrary triangulation of the $n$-gon, with $\Lambda_{Q_F(m,n)}$ and $\Lambda_{Q'}$ the corresponding dimer algebras and $e_b$ respectively $e_{b'}$ the sum of the boundary idempotents for $Q_F(m,n)$ and for $Q'$ respectively.
		Then there is an isomorphism
		\begin{align*}
		e_b \Lambda_{Q_F(m,n)} e_b \cong e_{b'} \Lambda_{Q'} e_{b'}.
		\end{align*}
	\end{theorem}
	\begin{proof}
		Combine the structure of the fan triangulation shown in Proposition \ref{prop:quiver_fan_general} with the proof of Theorem \ref{th:flip} for the general case.
	\end{proof}
	\begin{cor}\label{cor:central_element_general}
		Consider the boundary algebra $\mathcal{B}_Q$ of a dimer model $Q$ of a  GL$_m$-dimer of an arbitrary triangulation of the $n$-gon.
		Then the element $t$,
		\begin{align*}
		t := \sum_{i=1}^{m\cdot n}{u_i},
		\end{align*}
		is a central element of this algebra.
	\end{cor}
	\begin{proof}
		The element $t$ is the sum of exactly one chordless cycle for every boundary vertex and hence commutes with every element of $\mathcal{B}_Q$.
	\end{proof}
	
	Hence, as a consequence of Proposition \ref{prop:quiver_fan_general} and Theorem \ref{th:flip_equiv_general}, we can state the main result for the general case:
	
	\begin{theorem}[Main Theorem]
		Let $\mathcal{B}_Q$ be the boundary algebra obtained from the dimer model $Q$ of a  GL$_m$-dimer of an arbitrary triangulation of the $n$-gon.
		Then the quiver of $\mathcal{B}_Q$ with relations $\partial W$ is isomorphic to $\Gamma(m,n)$ subject to the following relations, for $k\in[1,m\cdot n]$ and indices are considered modulo $m\cdot n$:
		\begin{eqnarray*}
			x_{k+2-2i}y_{k} & = &  y_{k+1}z_{k}, \ k \not \equiv 0,1 \mod m, \ k \equiv -i \mod m, \text{ and } -i \in [1,m-2] \\
			x_{k+1}z_{k} &= & z_{k-1}x_{k}, \ k \not \equiv 0,1,2 \mod m \\
			x_{k+1} z_{k} &= & y_{k-2} x_{k-1}x_{k}, \ k \equiv 2 \mod m \\
			x_{k+1} x_{k+2} y_{k} &= & z_{k-1}x_{k}, \ k \equiv 0 \mod m \\
			y_{k+2-2j}y_{k} & = & x_{k+2m+1}x_{k+2m+2} \cdots x_{k} , \  k \equiv -j \mod m, \ k\not \equiv 1 \mod m, \ \text{ and } -j \in [0,m-2].
		\end{eqnarray*}
		Furthermore the element
		\begin{align*}
		t := \sum_{i=1}^{m\cdot n}{u_{i}}
		\end{align*}
		is central in $\mathcal{B}_Q$.
	\end{theorem}
	\begin{proof}
		Proposition \ref{prop:quiver_fan_general}, Lemma \ref{lm:flip_equivalence} and Theorem \ref{th:flip_equiv_general}, together with Corollary \ref{cor:central_element_general} give the desired result.
	\end{proof}
	
	
	\subsection*{Acknowledgements} This work was supported by the Austrian Science Fund (FWF) under Grant W1230, Doctoral Program ‘Discrete Mathematics’. He also wants to thank his supervisor Karin Baur for her great advice and support and an anonymous referee for carefully reading the manuscript and providing many helpful suggestions.


\begin{thebibliography}{Notes}
		
		\bibitem{bmt} Baur, K., Marsh, R., King, A. 
		{\em Dimer models and cluster categories of Grassmannians}, Proc. London Math. Soc. (2016) 113 (2): 213-260.
		
		\bibitem{bl}  Bocklandt, R. {\em Consistency conditions for dimer models}, Glasg. Math. J. 54 (2012), no. 2, 429-447
		
		
		\bibitem{g}  Goncharov, A.B. {\em Ideal webs, moduli spaces of local systems, and 3d {C}alabi-{Y}au categories}, Algebra, geometry, and physics in the 21st century, Progr. Math., 324, Birkh{\"a}user/Springer, Cham, (2017): 31-97. 
		
		\bibitem{gk}  Goncharov, A.B.,  Kenyon, R. {Dimers and Cluster and integrable systems}, Annales scientifiques de l'École Normale Supérieure 46.5 (2013): 747-813. 
		
		
		\bibitem{h}  Hatcher, A. {\em On triangulations of surfaces}, Topology and its Applications 40 (1991) 189-194, North-Holland
		
		
		\bibitem{oeis}  Sloane, N. J. A. editor, {\em The On-Line Encyclopedia of Integer Sequences}. Available at: \url{https://oeis.org}, \the\day .\the\month .\the\year
		
		
		\bibitem{p} Postnikov, A. {\em Total positivity, Grassmannians, and networks }, arXiv:math/0609764v1 [math.CO] Sep 2006.
		
	\end{thebibliography}
\end{document}